\documentclass[11pt,a4paper]{article}
\usepackage[margin=2cm]{geometry}
\RequirePackage[round, authoryear]{natbib}
\usepackage[labelfont=bf]{caption}
\usepackage[utf8]{inputenc}
\usepackage[T1]{fontenc}
\usepackage{array}
\usepackage{multirow}
\usepackage{amsmath}
\usepackage{amssymb}
\usepackage{stmaryrd}
\usepackage{graphicx}
\usepackage{wasysym}
\usepackage{diagbox}
\usepackage{amsthm}
\usepackage{enumerate}
\usepackage[shortlabels]{enumitem}
\usepackage[colorlinks,
linkcolor=blue,
anchorcolor=blue,
citecolor=blue]{hyperref}
\usepackage{float}
\usepackage{graphicx}
\usepackage{placeins}
\usepackage{makecell}
\usepackage{xr}
\usepackage{cases}
\usepackage{mathtools}
\usepackage[math]{cellspace}
\usepackage{bigints}
\usepackage{setspace}
\usepackage{enumitem}
\usepackage{times, cleveref}

\newcommand*\dif{\mathop{}\!\mathrm{d}}

\theoremstyle{definition}

\newtheorem{prop}{Proposition}

\title{ 
	\textbf{Real and complexified configuration spaces for planar 4-bar linkages}
}
	
\author{Zeyuan He, Simon D. Guest \\ \small Department of Engineering, University of Cambridge \\ \small zh299@cam.ac.uk, sdg@eng.cam.ac.uk}

\date{}

\begin{document}\maketitle
This note is a complete library of symbolic parametrized expressions for both real and complexified configuration spaces of a planar 4-bar linkage. Building upon the previous work from \citet{izmestiev_deformation_2015}, this library expands on the expressions by incorporating all four rotational angles across all possible linkage length choices, along with the polynomial relation between diagonals. Furthermore, a complete MATLAB app script \citet{he_elliptic-fun-based_2023} is included, enabling visualization and parametrization. The derivations are presented in a detailed manner, ensuring accessibility for researchers across diverse disciplines.

Planar 4-bar linkage is a fundamental component of single-degree-of-freedom mechanisms that have played a significant role in engineering for over two centuries. Extensive literature exists on its theories and applications. If focusing on the theoretical part, this note provides direct analytical calculations and examinations for all algebraic analyses pertaining to a planar 4-bar linkage. We refer to \citet{muller_novel_1996} for detailed illustration on each class of planar 4-bar linkage based on the number of independent linkage lengths and the application of planar 4-bar linkage in biological systems. In addition to our result, they provided discussion on the case when three of the linkage lengths are equal; \citet{gibson_geometry_1986} for the algebraic geometry of the coupler curve -- path traced by a specific point on the coupler link as it moves through its range of motion; \citet{torfason_double_1978} for details on the cusps, crunodes, double points and break points on corresponding elliptic curves; \citet{davies_finite_1979} for a lifting that maps the dimensions of planar 4-bar linkages onto finite regions in the three-dimensional space in-order for further parameter selection in engineering design; \citet{khimshiashvili_complex_2011} for a series of results on the moduli space, cross-ratio, generalized Heron polynomials used for calculating the area of the quadrilateral, etc.; \citet{chaudhary_balancing_2007} for a related linkage balancing problem in engineering application -- to reduce the amplitude of vibration and to smoothen highly fluctuating input-torque needed to obtain nearly constant drive speed.

In conjunction with these aforementioned contributions, the information provided in this note serves as a valuable handbook for engineers, architects, and mathematicians seeking parametrized symbolic expressions for configuring 4-bar linkage-based mechanisms.

\section{Modelling} \label{section: s1}

In this section we will set up the notation for a planar 4-bar linkage, as shown in Figure \ref{fig: planar 4-bar linkage}. We use $\alpha, ~\beta, ~\gamma, ~\delta \in \mathbb{R}^+$ to describe the \textbf{bar lengths} and four \textbf{rotational angles} $\rho_x, ~\rho_y, ~\rho_z, ~\rho_w \in \mathbb{R} \slash 2\pi$ to describe the configuration. $\mathbb{R} \slash 2\pi$ means $a, ~b \in \mathbb{R}$ are equivalent if $b = a + 2k\pi, ~k \in \mathbb{Z}$, which implies that we take $-\pi$ and $\pi$ as the same rotational angle.

The tangents of half of the rotational angles are defined as 
\begin{equation}
	\begin{gathered}
		x=\tan \dfrac{\rho_x}{2}, ~~y=\tan \dfrac{\rho_y}{2}, ~~z=\tan \dfrac{\rho_z}{2}, ~~w=\tan \dfrac{\rho_w}{2} \\
		x, ~y, ~z, ~w \in \mathbb{R} \cup \{\infty\}
	\end{gathered}
\end{equation}
$\mathbb{R} \cup \{\infty\}$ is the \textbf{one-point compactification} of the real line, which `glues' $-\infty$ and $\infty$ to the same point. Moreover, in further algebraic discussions. $u$ and $v$ are the lengths of diagonals.

\begin{figure} [!tb]
	\noindent \begin{centering}
		\includegraphics[width=0.5\linewidth]{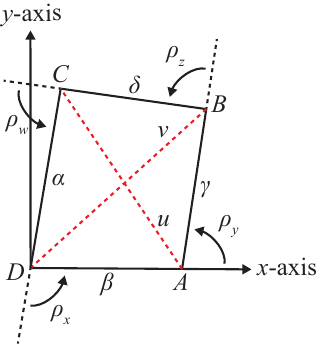}
		\par \end{centering}
	
	\caption{\label{fig: planar 4-bar linkage}(a) A planar 4-bar linkage. We label the linkages lengths counter-clockwise as $\alpha$, $\beta$, $\gamma$, $\delta$, and the rotational angles counter-clockwise as $\rho_x$, $\rho_y$, $\rho_z$, $\rho_w$. The tangents of half of these folding angles are $x, ~y, ~z, ~w$. Then $u$ and $v$ are the lengths of two diagonals of the quadrilateral,}
\end{figure}

First, we have the following proposition on the range of linkage lengths.
\begin{prop} \label{prop: linkage length range}
	$\alpha, ~\beta, ~\gamma, ~\delta \in \mathbb{R}^+$ are the bar lengths of a planar 4-bar linkage if and only if
	\begin{equation} \label{eq: linkage length range}
		\begin{cases}
			\begin{gathered}
				\alpha < \beta + \gamma + \delta \\
				\beta < \gamma + \delta + \alpha \\
				\gamma < \delta + \alpha + \beta \\
				\delta < \alpha + \beta + \gamma 
			\end{gathered}
		\end{cases}
	\end{equation}
\end{prop}

\begin{proof}
	Necessity: for a planar quadrilateral (possibly self-intersected), Equation \eqref{eq: linkage length range} holds.
	
	Sufficiency: consider the four triangles $(\alpha,~ \beta, ~u)$, $(\gamma, ~\delta, ~u)$, $(\beta, ~ \gamma, ~v)$ and $(\delta, ~\alpha, ~v)$ when $\alpha, ~\beta, ~\gamma, ~\delta \in \mathbb{R}^+$. $u$ and $v$ should satisfy the following condition:
	\begin{equation*}
		\begin{cases}
			\begin{gathered}
				|\alpha-\beta| < u < \alpha+\beta \\
				|\gamma-\delta| < u < \gamma+\delta \\
				|\beta-\gamma| < v < \beta+\gamma \\
				|\delta-\alpha| < v < \delta+\alpha 
			\end{gathered}
		\end{cases}
	\end{equation*}
	The sufficient condition for $\alpha, ~\beta, ~ \gamma, ~\delta$ to be the linkage lengths is the range of $u$ and $v$ being non-empty, which means:
	\begin{equation*}
		\begin{cases}
			\begin{gathered}
				|\alpha-\beta|  < \gamma+\delta \\
				|\gamma-\delta| < \alpha+\beta \\
				|\beta-\gamma| < \delta+\alpha \\
				|\delta-\alpha| < \beta+\gamma \\ 
			\end{gathered}
		\end{cases}
	\end{equation*}
	This leads to Equation \eqref{eq: linkage length range}.
\end{proof}

Next we will give the polynomial relation between:
\begin{enumerate} [label={[\arabic*]}]
	\item adjacent rotational angles $x$ and $y$; $x$ and $w$.
	\item opposite rotational angles $x$ and $z$.
	\item lengths of diagonals $u$ and $v$.
\end{enumerate}
then conduct a detailed case-by-case discussion.

\section{Relation between adjacent rotational angles}

One method to derive the relation between $x$ and $y$ is to build an Euclidean coordinate, as shown in Figure \ref{fig: planar 4-bar linkage}. We have:
\begin{equation*}
	\begin{gathered}
		B \left(\beta+\gamma \cos \rho_y, ~~ \gamma \sin \rho_y \right) \\
		C \left(-\alpha \cos \rho_x, ~~ \alpha \sin \rho_x \right) 
	\end{gathered}
\end{equation*}
Knowing the distance of $BC$, we could obtain the relation between $x$ and $y$:
\begin{equation*}
	\begin{aligned}
		(\beta+\gamma \cos \rho_y+\alpha \cos \rho_x)^2+(\gamma \sin \rho_y-\alpha \sin \rho_x)^2=\delta^2
	\end{aligned}
\end{equation*}
Rewrite the above in a rational form:
\begin{equation*}
	\begin{aligned}
		\left( \beta+\gamma \dfrac{1-y^2}{1+y^2}+\alpha \dfrac{1-x^2}{1+x^2} \right)^2+\left(\gamma \dfrac{2y}{1+y^2}-\alpha \dfrac{2x}{1+x^2} \right)^2=\delta^2
	\end{aligned}
\end{equation*}
The next step is to simplify the above to a polynomial form below:
\begin{equation} \label{eq: planar 4-bar linkage}
	f(\alpha, ~\beta, ~\gamma, ~\delta, ~x, ~y) = f_{22} x^2y^2 + f_{20}x^2 + 2f_{11}xy + f_{02}y^2 + f_{00} = 0
\end{equation}
\begin{equation*}
	\begin{gathered}
		f_{22} = \dfrac{(\alpha-\beta+\gamma+\delta)(\alpha-\beta+\gamma-\delta)}{4} = (\sigma-\beta)(\sigma-\beta-\delta)  \\
		f_{20} = \dfrac{(-\alpha+\beta+\gamma+\delta)(-\alpha+\beta+\gamma-\delta)}{4} = (\sigma-\alpha)(\sigma-\alpha-\delta) \\
		f_{11} = - \alpha \gamma \\
		f_{02} = \dfrac{(\alpha+\beta-\gamma-\delta)(\alpha+\beta-\gamma+\delta)}{4} =  (\sigma-\gamma)(\sigma-\gamma-\delta) \\
		f_{00} = \dfrac{(\alpha+\beta+\gamma+\delta)(\alpha+\beta+\gamma-\delta)}{4} = 
		\sigma(\sigma-\delta) \\
		\sigma=\dfrac{\alpha+\beta+\gamma+\delta}{2}
	\end{gathered}
\end{equation*}
$\sigma$ is the \textit{semi-perimeter} of a planar 4-bar linkage. Note that $\alpha, ~\beta, ~\gamma, ~\delta$ satisfy Proposition \ref{prop: linkage length range}, and we allow $x, ~y \in \mathbb{R} \cup \{\infty\}$ in Equation \eqref{eq: planar 4-bar linkage}. 

Further we will classify Equation \eqref{eq: planar 4-bar linkage} based on the \textbf{degree of degeneracy} -- how many of $f_{22}, ~f_{20}, ~f_{02}$ are zero. Clearly $f_{11}<0$, and from Proposition \ref{prop: linkage length range} we have:
\begin{equation*}
	\sigma-\alpha \neq 0, ~~ \sigma-\beta \neq 0, ~~ \sigma-\gamma \neq 0, ~~ \sigma-\delta \neq 0
\end{equation*}
so $f_{00}>0$. We classify a planar 4-bar linkage with the same terminologies provided in previous references:
\begin{enumerate} [label={[\arabic*]}]
	\item a \textit{rhombus}, if:
	\begin{equation*}
		f_{22}=0, ~~ f_{20}=0,~~f_{02}=0 ~~ \Leftrightarrow ~~  \alpha=\beta=\gamma=\delta
	\end{equation*}
	\item an \textit{isogram}, if: 
	\begin{equation*}
		f_{22} \neq 0, ~~ f_{20}=0,~~f_{02}=0 ~~ \Leftrightarrow ~~ \gamma=\alpha, ~~ \delta=\beta, ~~\beta \neq \alpha
	\end{equation*}
	\item a \textit{deltoid I}, if:
	\begin{equation*}
		f_{22}=0, ~~ f_{20} \neq 0, ~~f_{02}=0 ~~ \Leftrightarrow ~~ \delta=\alpha, ~~ \gamma=\beta, ~~\beta \neq \alpha
	\end{equation*}
	\item a \textit{deltoid II}, if:
	\begin{equation*}
		f_{22}=0, ~~ f_{20}=0 , ~~ f_{02} \neq 0 ~~ \Leftrightarrow ~~ \alpha=\beta, ~~ \delta=\gamma, ~~\gamma \neq \beta
	\end{equation*}
	\item of \textit{type conic I}, if:
	\begin{equation*} 
		f_{22}=0, ~~ f_{20} \neq 0 , ~~ f_{02} \neq 0  ~~ \Leftrightarrow ~~ \begin{dcases}
			\alpha - \beta + \gamma - \delta = 0 \\
			\alpha - \beta - \gamma + \delta \neq 0 \\
			\alpha + \beta - \gamma - \delta \neq 0
		\end{dcases}
	\end{equation*}
	\item of \textit{type conic II}, if:
	\begin{equation*} 
		f_{22} \neq 0, ~~ f_{20} = 0 , ~~ f_{02} \neq 0 ~~ \Leftrightarrow ~~ \begin{dcases}
			\alpha - \beta + \gamma - \delta \neq 0 \\
			\alpha - \beta - \gamma + \delta = 0 \\
			\alpha + \beta - \gamma - \delta \neq 0
		\end{dcases}
	\end{equation*}
	\item of \textit{type conic III}, if:
	\begin{equation*} 
		f_{22} \neq 0, ~~ f_{20} \neq 0 , ~~ f_{02} = 0 ~~ \Leftrightarrow ~~ \begin{dcases}
			\alpha - \beta + \gamma - \delta \neq 0 \\
			\alpha - \beta - \gamma + \delta \neq 0 \\
			\alpha + \beta - \gamma - \delta = 0
		\end{dcases}
	\end{equation*}
	\item of type \textit{elliptic}, if:
	\begin{equation*} 
		f_{22} \neq 0, ~~ f_{20} \neq 0 , ~~ f_{02} \neq 0 ~~ \Leftrightarrow ~~ \begin{dcases}
			\alpha - \beta + \gamma - \delta \neq 0 \\
			\alpha - \beta - \gamma + \delta \neq 0 \\
			\alpha + \beta - \gamma - \delta \neq 0
		\end{dcases}
	\end{equation*}
	Further, if the following relation is satisfied:
	\begin{equation*}
		\alpha^2 + \gamma^2 = \beta^2 + \delta^2
	\end{equation*}
	We call this type \textit{orthodiagonal} since the diagonals are orthogonal, and we will analyze its special configuration space later.
\end{enumerate}
Similarly we have the following relation over $x$ and $w$:
\begin{equation} \label{eq: planar 4-bar linkage 2}
	f( \beta, ~\alpha, ~\delta, ~\gamma, ~x, ~w) = 0
\end{equation} 

\section{Relation between opposite rotational angles}

We could then derive the relation between $z^2$ and $x^2$:
\begin{equation*}
	\begin{gathered}
		u^2 = \alpha^2 + \beta^2 +2 \alpha \beta \cos \rho_x = \delta^2 + \gamma^2 +2 \delta \gamma \cos \rho_z  \\
		\Rightarrow \alpha^2 + \beta^2 +2 \alpha \beta \dfrac{1-x^2}{1+x^2} = \delta^2 + \gamma^2 +2 \delta \gamma \dfrac{1-z^2}{1+z^2} 
	\end{gathered}
\end{equation*}
The next step is to transfer this equation to a polynomial:
\begin{equation} \label{eq: opposite rotational angle}
	g(\alpha, ~\beta, ~\gamma, ~\delta, ~x, ~ z)=g_{22}x^2z^2+g_{20}x^2+g_{02}z^2+g_{00}=0
\end{equation}
\begin{equation*} 
	\begin{gathered}
		g_{22}= \dfrac{(\alpha-\beta+\gamma-\delta)(-\alpha+\beta+\gamma-\delta)}{4} = (\sigma-\alpha-\gamma)(\sigma-\beta-\gamma) \\
		g_{20} = \dfrac{(- \alpha + \beta + \gamma + \delta)(\alpha - \beta + \gamma + \delta)}{4} = (\sigma-\alpha) (\sigma-\beta) \\
		g_{02} = \dfrac{(-\alpha - \beta + \gamma - \delta)(\alpha + \beta + \gamma - \delta)}{4} = -(\sigma-\gamma) (\sigma-\delta) \\
		g_{00} = \dfrac{(\alpha + \beta + \gamma + \delta)(-\alpha -\beta + \gamma + \delta)}{4} = \sigma (\sigma - \alpha - \beta) \\
		\sigma=\dfrac{\alpha+\beta+\gamma+\delta}{2}
	\end{gathered}
\end{equation*}
The identities [1]--[4] listed in Proposition \ref{prop: identities} would help to understand the symmetry of coefficients in Equation \eqref{eq: opposite rotational angle}.

It is clear that $g_{20}>0$ and $g_{02}<0$. Note that $x$ and $z$ can be separated if $g_{22}x^2+g_{02} \neq 0$:
\begin{equation}
	z^2=-\dfrac{g_{20}x^2+g_{00}}{g_{22}x^2+g_{02}}
\end{equation}
Equation \eqref{eq: opposite rotational angle} after a case-by-case study is:
\begin{enumerate} [label={[\arabic*]}]
	\item rhombus
	\begin{equation*}
		\begin{aligned}
			\alpha=\beta=\gamma=\delta & ~~ \Rightarrow ~~ g_{22}=0, ~~ g_{20}=-g_{02}=\alpha^2, ~~ g_{00}=0 \\
			& ~~ \Rightarrow ~~ z = \pm x
		\end{aligned}
	\end{equation*}
	\item isogram 
	\begin{equation*}
		\begin{aligned}
			\gamma=\alpha, ~~ \delta=\beta, ~~\beta \neq \alpha & ~~ \Rightarrow ~~ g_{22}=0, ~~ g_{20}=-g_{02}=\alpha\beta, ~~ g_{00}=0 \\
			& ~~ \Rightarrow ~~ z= \pm x
		\end{aligned}
	\end{equation*}
	\item deltoid I
	\begin{equation*}
		\begin{aligned}
			\delta=\alpha, ~~ \gamma=\beta, ~~\beta \neq \alpha & ~~ \Rightarrow ~~ g_{22}=0, ~~ g_{20}=-g_{02}=\alpha\beta, ~~ g_{00}=0 \\
			& ~~ \Rightarrow ~~ z= \pm x
		\end{aligned}
	\end{equation*}
	\item deltoid II
	\begin{equation*}
		\begin{aligned}
			\alpha=\beta, ~~ \delta=\gamma, ~~\gamma \neq \beta & ~~ \Rightarrow ~~ g_{22}=0, ~~ g_{20}=\gamma^2, ~~ g_{02}=-\beta^2, ~~ g_{00}=\gamma^2-\beta^2 \\
			& ~~ \Rightarrow ~~ \beta^2 (z^2 +1) = \gamma^2 (x^2+1)
		\end{aligned}
	\end{equation*}
	\item conic I
	\begin{equation*}
		\begin{aligned}
			\begin{dcases}
				\alpha - \beta + \gamma - \delta = 0 \\
				\alpha - \beta - \gamma + \delta \neq 0 \\
				\alpha + \beta - \gamma - \delta \neq 0
			\end{dcases} & ~~ \Rightarrow ~~ \begin{cases}
				g_{22}=0, ~~ g_{20}=\gamma\delta, ~~ g_{02}=-\alpha\beta, \\  g_{00}=(\alpha+\gamma)(-\beta+\gamma)= -\alpha\beta+\gamma\delta
			\end{cases} \\
			& ~~ \Rightarrow ~~ \alpha \beta (z^2+1) = \gamma\delta (x^2 +1 )
		\end{aligned}
	\end{equation*}
	\item conic II
	\begin{equation*}
		\begin{aligned}
			\begin{dcases}
				\alpha - \beta + \gamma - \delta \neq 0 \\
				\alpha - \beta - \gamma + \delta = 0 \\
				\alpha + \beta - \gamma - \delta \neq 0
			\end{dcases} & ~~ \Rightarrow ~~ \begin{cases}
				g_{22}=0, ~~ g_{20}=\gamma\delta, ~~ g_{02}=-\alpha\beta, \\  g_{00}= (-\alpha+\gamma)(\beta+\gamma)=-\alpha\beta+\gamma\delta
			\end{cases} \\
			& ~~ \Rightarrow ~~ \alpha \beta (z^2+1) = \gamma\delta (x^2 +1 )
		\end{aligned}
	\end{equation*}
	\item conic III
	\begin{equation*} 
		\begin{aligned}
			\begin{dcases}
				\alpha - \beta + \gamma - \delta \neq 0 \\
				\alpha - \beta - \gamma + \delta \neq 0 \\
				\alpha + \beta - \gamma - \delta = 0
			\end{dcases} & ~~ \Rightarrow ~~ \begin{cases}
				g_{22}= (\beta-\gamma)(\alpha-\gamma) =\alpha\beta-\gamma\delta, \\
				g_{20}=\alpha\beta, ~~ g_{02}=-\gamma\delta, ~~ g_{00}=0
			\end{cases} \\
			& ~~ \Rightarrow ~~ \alpha \beta (z^{-2}+1) = \delta\gamma (x^{-2} +1 )
		\end{aligned}
	\end{equation*}
	\item Elliptic. See Section \ref{section: planar elliptic} since the expression is not reduced.
\end{enumerate}

\section{Relation between the lengths of diagonals}

The polynomial relation between the lengths of diagonals $u$ and $v$ is also of interest to us and is actually used quite often.

For simplicity, we need to consider the relation between $u$ and $v$ without involving the folding angles. A great tool is called the \textit{Cayley--Menger determinant} over the four points of a planar 4-bar linkage. The coplanarity is equivalent to
\begin{equation*}
	\begin{gathered}
		\left|
		\begin{matrix}
			0 & 1 & 1 & 1 & 1 \\
			1 & 0 & \alpha^2 & u^2 & \delta^2 \\
			1 & \alpha^2 & 0 & \beta^2 & v^2 \\
			1 & u^2 & \beta^2 & 0 & \gamma^2 \\
			1 & \delta^2 & v^2 & \gamma^2 & 0
		\end{matrix}
		\right| = 0 \\[10pt]
	\end{gathered}
\end{equation*}
After full simplification, the result is an order 3 polynomial over $u^2$ and $v^2$:
\begin{equation} \label{eq: relation between diagonals planar}
	h(\alpha, ~\beta, ~\gamma, ~\delta, ~u, ~ v)=u^4v^2+u^2v^4+h_{11}u^2v^2+h_{10}u^2+h_{01}v^2+h_{00}=0
\end{equation}
where
\begin{equation*}
	\begin{gathered}
		h_{11}=-(\alpha^2+\beta^2+\gamma^2+\delta^2) \\
		h_{10}=-(\beta^2-\gamma^2)(\delta^2-\alpha^2) \\
		h_{01}=-(\alpha^2-\beta^2)(\gamma^2-\delta^2) \\
		h_{00}=(\alpha^2 \gamma^2-\beta^2 \delta^2) (\alpha^2-\beta^2+\gamma^2-\delta^2) \\
	\end{gathered}
\end{equation*}
Here we use five simple examples to show how to make use of Equation \eqref{eq: relation between diagonals planar}:
\begin{enumerate} [label={[\arabic*]}]
	\item rhombus
	\begin{equation*}
		\begin{aligned}
			\alpha=\beta=\gamma=\delta & ~~ \Rightarrow ~~ h_{11}=-4\alpha^2, ~~ h_{10}=h_{01}=h_{00}=0 \\
			& ~~ \Rightarrow ~~ u^2+v^2=4\alpha^2
		\end{aligned}
	\end{equation*}
	\item isogram 
	\begin{equation*}
		\begin{aligned}
			\gamma=\alpha, ~~ \delta=\beta  ~~ \Rightarrow ~~ & h_{11}=-2(\alpha^2+\beta^2), ~~ h_{10}=h_{01}=-(\alpha^2-\beta^2)^2 \\ & h_{00}=2(\alpha^4-\beta^4)(\alpha^2-\beta^2) \\
			~~ \Rightarrow ~~ & (uv-\alpha^2+\beta^2) (uv+\alpha^2-\beta^2) (u^2+v^2-2\alpha^2-2\beta^2)=0 \\
			~~ \Rightarrow ~~ & u^2+v^2=2\alpha^2+2\beta^2
		\end{aligned}
	\end{equation*}
	since $uv>\alpha^2$ and $uv>\beta^2$.
	\item deltoid I
	\begin{equation*}
		\begin{aligned}
			\delta=\alpha, ~~ \gamma=\beta & ~~ \Rightarrow ~~ h_{11}=-2(\alpha^2+\beta^2), ~~ h_{10}=0, ~~ h_{01}=(\alpha^2-\beta^2)^2, ~~ h_{00}=0 \\
			& ~~ \Rightarrow ~~ u^4+u^2v^2-2(\alpha^2+\beta^2)u^2+(\alpha^2-\beta^2)^2=0 \\
			& ~~ \Rightarrow ~~ v^2 = 2\alpha^2+2\beta^2-u^2-\dfrac{(\alpha^2-\beta^2)^2}{u^2}
		\end{aligned}
	\end{equation*}
	Here we could see that $v^2$ reaches the maximum $2\alpha^2+2\beta^2-2|\alpha^2-\beta^2|$ when $u^2=|\alpha^2-\beta^2|$.
	\item deltoid II
	\begin{equation*}
		\begin{aligned}
			\alpha=\beta, ~~ \delta=\gamma & ~~ \Rightarrow ~~ h_{11}=-2(\beta^2+\gamma^2), ~~ h_{10}=(\beta^2-\gamma^2)^2, ~~ h_{01}=0, ~~ h_{00}=0 \\
			& ~~ \Rightarrow ~~ v^4+v^2u^2-2(\beta^2+\gamma^2)v^2+(\beta^2-\gamma^2)^2=0 \\
			& ~~ \Rightarrow ~~ u^2 = 2\beta^2+2\gamma^2-v^2-\dfrac{(\beta^2-\gamma^2)^2}{v^2}
		\end{aligned}
	\end{equation*}
	Here we could see that $u^2$ reaches the maximum $2\beta^2+2\gamma^2-2|\beta^2-\gamma^2|$ when $v^2=|\beta^2-\gamma^2|$.
	\item Orthodiagonal
	\begin{equation*}
		\begin{gathered}
			h_{11}=-2(\alpha^2+\gamma^2) \\
			h_{10}=(\beta^2-\gamma^2)^2 \\
			h_{01}=(\alpha^2-\beta^2)^2 \\
			h_{00}=0 \\
		\end{gathered}
	\end{equation*}
\end{enumerate}
hence $u$ and $v$ are separable:
\begin{equation*}
	h(\alpha, ~\beta, ~\gamma, ~\delta, ~u, ~ v)=u^2+\dfrac{(\beta^2-\gamma^2)^2}{u^2}+v^2+\dfrac{(\alpha^2-\beta^2)^2}{v^2}-2(\alpha^2+\gamma^2) = 0
\end{equation*}

An important conclusion below is derived from Equation \eqref{eq: relation between diagonals planar}:
\begin{prop}
	There is a one-to-one correspondence between the two planar 4-bar linkages below:
	\begin{equation}
		(\alpha, ~\beta, ~\gamma, ~\delta, ~u, ~v) \Leftrightarrow (\sigma-\alpha, ~\sigma-\beta, ~\sigma-\gamma, ~\sigma-\delta, ~u, ~v)
	\end{equation}
	That is to say, for every planar 4-bar linkage with linkage lengths $\alpha, ~\beta, ~\gamma, ~\delta$ and diagonal lengths $u, ~v$, there is another planar 4-bar linkage with linkage lengths $\sigma-\alpha, ~\sigma-\beta, ~\sigma-\gamma, ~\sigma-\delta$ and the same diagonal lengths $u, ~v$. We say they are \textit{conjugate} planar 4-bar linkages.
\end{prop}

Actually, from direct symbolic calculations with the helpful identities provided in Proposition \ref{prop: identities} below, we could see that the coefficients $h_{11}, ~h_{10}, ~h_{01}, ~h_{00}$ preserve when switching a planar 4-bar linkage to its conjugate.

\section{Helpful identities, sign convention and switching a strip} \label{section: sign convention}

In this section we will provide some helpful identities in the study of the above polynomial relations.

\begin{prop} \label{prop: identities}
	Some identities over $\alpha, ~\beta, ~\gamma, ~ \delta$ and $\sigma-\alpha, ~\sigma-\beta, ~ \sigma-\gamma, ~\sigma-\delta$. Note that $\sigma=(\alpha+\beta+\gamma+\delta)/2$ is the semi-perimeter.
	\begin{enumerate} [label={[\arabic*]}]
		\item 
		\begin{equation*}
			\sigma-\alpha-\beta =  -(\sigma-\gamma-\delta)
		\end{equation*}
		\item
		\begin{equation*}
			\begin{aligned}
				(\sigma-\alpha)(\sigma-\beta)-\alpha\beta & = \sigma(\sigma-\alpha-\beta) \\
				& = -\sigma(\sigma-\gamma-\delta) \\
				& =  \gamma\delta-(\sigma-\gamma)(\sigma-\delta)
			\end{aligned}
		\end{equation*} 
		\item
		\begin{equation*}
			\begin{aligned}
				(\sigma-\alpha)(\sigma-\beta)-\gamma\delta & = (\sigma-\alpha)(\sigma-\beta)-(\gamma+\delta)\gamma+ \gamma^2 \\
				& = (\sigma-\alpha)(\sigma-\beta)-(2\sigma-\alpha-\beta)\gamma+ \gamma^2 \\
				& = (\sigma-\alpha-\gamma)(\sigma-\beta-\gamma) \\
				& = -(\sigma-\gamma-\alpha)(\sigma-\delta-\alpha) \\
				& = \alpha\beta-(\sigma-\gamma)(\sigma-\delta)
			\end{aligned}
		\end{equation*}
		\item 
		\begin{equation*}
			\begin{gathered}
				\alpha + \beta = (\sigma-\gamma) + (\sigma - \delta) \\
				\alpha - \beta = (\sigma-\beta) - (\sigma - \alpha) \\
				\alpha \beta + \gamma \delta = (\sigma-\alpha)(\sigma-\beta) + (\sigma-\gamma)(\sigma-\delta)
			\end{gathered} 
		\end{equation*}
		\item 
		\begin{equation*}
			\begin{aligned}
				\alpha \beta \gamma \delta - (\sigma-\alpha)(\sigma-\beta)(\sigma-\gamma)(\sigma-\delta) & = [(\sigma-\alpha)(\sigma-\beta)+(\sigma-\gamma)(\sigma-\delta)-\gamma\delta]\gamma\delta \\ & \quad -(\sigma-\alpha)(\sigma-\beta)(\sigma-\gamma)(\sigma-\delta) \\ 
				& = ((\sigma-\alpha)(\sigma-\beta)-\gamma\delta)(\gamma\delta-(\sigma-\gamma)(\sigma-\delta)) \\
				& = \sigma(\sigma-\alpha-\beta)(\sigma-\alpha-\gamma)(\sigma-\beta-\gamma)			
			\end{aligned}
		\end{equation*}
		\item 
		\begin{equation*}
			\begin{aligned}
				\alpha^2+\beta^2+\gamma^2+\delta^2 & = (\alpha + \beta + \gamma + \delta)^2- 2 \alpha \beta - 2 \alpha \gamma - 2 \alpha \delta -2 \beta \gamma -2 \beta \delta -2 \gamma \delta  \\
				& = [(\sigma-\alpha) + (\sigma-\beta) + (\sigma-\gamma) + (\sigma-\delta)]^2 \\ 
				& \quad ~~ - 2 (\sigma-\alpha) (\sigma-\beta) - 2 (\sigma-\alpha) (\sigma-\gamma) - 2 (\sigma-\alpha) (\sigma-\delta) \\ 
				& \quad ~~ -2 (\sigma-\beta) (\sigma-\gamma) -2 (\sigma-\beta) (\sigma-\delta) -2 (\sigma-\gamma) (\sigma-\delta) \\
				& =  (\sigma-\alpha)^2+(\sigma-\beta)^2+(\sigma-\gamma)^2+(\sigma-\delta)^2
			\end{aligned}
		\end{equation*}
		\item 
		\begin{equation*}
			\begin{aligned}
				\alpha \beta - \gamma \delta & = \dfrac{1}{2}[(\alpha+\beta)^2+(\gamma-\delta)^2-\alpha^2-\beta^2-\gamma^2-\delta^2] \\
				& = \dfrac{1}{2}[(\sigma-\gamma+\sigma-\delta)^2+(\sigma-\gamma-\sigma+\delta)^2-\alpha^2-\beta^2-\gamma^2-\delta^2] \\
				& = \dfrac{1}{2}[2(\sigma-\gamma)^2+2(\sigma-\delta)^2-\alpha^2-\beta^2-\gamma^2-\delta^2] \\
				& = \dfrac{1}{2}[-(\sigma-\alpha)^2-(\sigma-\beta)^2+(\sigma-\gamma)^2+(\sigma-\delta)^2]
			\end{aligned}
		\end{equation*}
		\item When $\alpha^2 + \gamma^2 = \beta^2 + \delta^2$:
		\begin{equation*}
			\begin{gathered}
				\begin{aligned}
					(\sigma-\beta)(\sigma-\beta-\delta) & = \dfrac{1}{4} \left((\alpha+\gamma-\beta)^2 - \delta^2 \right) \\
					& = \dfrac{1}{4} \left((\alpha+\gamma)^2-2\beta(\alpha+\gamma)+\beta^2 - \delta^2 \right) \\
					& = \dfrac{1}{2}(\beta - \alpha)(\beta - \gamma)
				\end{aligned} \\
				(\sigma-\alpha)(\sigma-\alpha-\delta) = \dfrac{1}{2}(\beta - \alpha)(\beta + \gamma) \\
				(\sigma-\gamma)(\sigma-\gamma-\delta) = \dfrac{1}{2}(\beta + \alpha)(\beta - \gamma) \\
				\sigma(\sigma-\delta) = \dfrac{1}{2}(\beta + \alpha)(\beta + \gamma) \\
				(\sigma - \alpha)(\sigma - \gamma) = \dfrac{1}{2}(\alpha \gamma + \beta \delta) = (\sigma - \beta)(\sigma - \delta)
			\end{gathered}
		\end{equation*}
	\end{enumerate}
	These identities also hold for all the permutations over $\alpha, ~\beta, ~\gamma, ~\delta$.
\end{prop}
We will follow the sign convention below when calculating the square root.
\begin{equation} \label{eq: sqrt sign convention planar}
	\sqrt{a}=
	\begin{cases}
		\begin{aligned}
			& \sqrt{a} ~(a \ge 0) \\
			& i \sqrt{-a} ~(a<0)
		\end{aligned}
	\end{cases}
\end{equation}
At the end of this section we introduce the following important transformation called `switching a strip' (this name is motivated from what Prof. Ivan Izmestiev proposed for a spherical 4-bar linkage):
\begin{prop} \label{prop: switch a strip planar}
	One-to-one correspondence between two spherical quadrilaterals if allowing the linkage length to be negative. They actually have the same trajectory when drawing on a plane.
	\begin{enumerate} [label={[\arabic*]}]
		\item $(\alpha, ~\beta, ~\gamma, ~\delta, ~x, ~y, ~z, ~w) \Leftrightarrow (\alpha, ~\beta, ~-\gamma, ~-\delta, ~x, ~-y^{-1}, ~-z, ~-w^{-1})$
		\item $(\alpha, ~\beta, ~\gamma, ~\delta, ~x, ~y, ~z, ~w) \Leftrightarrow (-\alpha, ~\beta, ~\gamma, ~-\delta, ~-x^{-1}, ~y, ~-z^{-1}, ~-w)$ 
		\item $(\alpha, ~\beta, ~\gamma, ~\delta, ~x, ~y, ~z, ~w) \Leftrightarrow (-\alpha, ~ - \beta, ~\gamma, ~\delta, ~-x, ~-y^{-1}, ~z, ~-w^{-1})$ 
		\item $(\alpha, ~\beta, ~\gamma, ~\delta, ~x, ~y, ~z, ~w) \Leftrightarrow (\alpha, ~ - \beta, ~ - \gamma, ~\delta, ~-x^{-1}, ~-y, ~-z^{-1}, ~w)$ 
	\end{enumerate}
\end{prop}

\section{Post-examination, solutions at infinity}

Generically, given a value of $x$ we will have two solutions of $y$ from Equation \eqref{eq: planar 4-bar linkage}, two solutions of $z$ from Equation \eqref{eq: planar 4-bar linkage 2}, and two solutions of $w$ from Equation \eqref{eq: opposite rotational angle}. We need to figure out which tuple of $x, ~y, ~z, ~w$ is the actual solution among the possible 8 choices. The most efficient way, as far as we know, would be post-examination using the relations between $z$ and $y$; $z$ and $w$; and $y$ and $w$. They can be written explicitly by applying a permutation on $\alpha, ~\beta, ~\gamma, ~\delta$.

Our discussion in previous sections are not sufficient when talking about solutions at infinity for a polynomial system. In order for a complete discussion, it is necessary to bihomogenize Equations \eqref{eq: planar 4-bar linkage}, \eqref{eq: planar 4-bar linkage 2} and \eqref{eq: opposite rotational angle}.

Let 
\begin{equation*}
	x=\dfrac{x_1}{x_2}, ~~ y=\dfrac{y_1}{y_2} 
\end{equation*}
then equations \eqref{eq: planar 4-bar linkage}, \eqref{eq: planar 4-bar linkage 2} and \eqref{eq: opposite rotational angle} become
\begin{equation} \label{eq: planar 4-bar linkage homo}
	f(\alpha, ~\beta, ~\gamma, ~\delta, ~x_1, ~x_2, ~y_1, ~y_2) = f_{22} x_1^2y_1^2 + f_{20} x_1^2y_2^2 + 2f_{11} x_1y_1x_2y_2 + f_{02} x_1^2y_2^2 + f_{00} x_2^2y_2^2
\end{equation}
\begin{equation} \label{eq: opposite rotational angle homo}
	g(\alpha, ~\beta, ~\gamma, ~\delta, ~x_1, ~x_2, ~z_1, ~z_2)=g_{22} x_1^2z_1^2 + g_{20} x_1^2z_2^2 + g_{02} z_1^2x_2^2 + g_{00}x_2^2z_2^2=0
\end{equation}
\begin{equation*}
	x_1, ~x_2, ~y_1, ~y_2, ~z_1, ~z_2, ~w_1, ~w_2 \in \mathbb{R}
\end{equation*}
For example, to check if there is a solution of $x$ at infinity, we could let $x_1 \neq 0, ~~ x_2=0$. 

The sections below will be a case-by-case discussion on the finite solutions and solutions at infinity. 

\section{Rhombus: finite solution}

The condition on linkage lengths is 
\begin{equation*}
	f_{22}=0, ~~ f_{20}=0,~~f_{02}=0 ~~ \Leftrightarrow ~~  \alpha=\beta=\gamma=\delta
\end{equation*}
which means:
\begin{equation*}
	\begin{aligned}
		\begin{dcases}
			xy = 1 \\
			z = \pm x \\
			xw = 1 
		\end{dcases} ~~ \Rightarrow ~~ & \begin{dcases}
			y = \dfrac{1}{x} \\
			z = \pm x \\
			w = \dfrac{1}{x} 
		\end{dcases}
	\end{aligned}
\end{equation*}
From post-examination, there is a single branch of flex diffeomorphic to a circle $S^1$ with the parametrization below:
\begin{equation*}
	\begin{dcases}
		x \in \mathbb{R} \cup \{\infty\} \\
		y = \dfrac{1}{x} \\
		z = x \\
		w = \dfrac{1}{x} 
	\end{dcases} 
\end{equation*}
This branch exactly has the shape of a rhombus.

\section{Isogram: finite solution}

The condition on linkage lengths is 	
\begin{equation*}
	f_{22} \neq 0, ~~ f_{20}=0,~~f_{02}=0 ~~ \Leftrightarrow ~~ \gamma=\alpha, ~~ \delta=\beta, ~~\beta \neq \alpha
\end{equation*}
which means:
\begin{equation*}
	\begin{aligned}
		\begin{cases}
			(\alpha-\beta)x^2y^2-2\alpha xy+(\alpha+\beta)=0 \\
			z = \pm x \\
			(\beta-\alpha)x^2z^2-2\beta xw+(\alpha+\beta) =0 
		\end{cases} 
		~~ \Rightarrow ~~ \begin{dcases}
			y = \dfrac{1}{x} ~~ \mathrm{or} ~~ y = \dfrac{\alpha+\beta}{(\alpha-\beta)x} \\
			z = \pm x \\
			w = \dfrac{1}{x} ~~ \mathrm{or} ~~ w = \dfrac{\beta+\alpha}{(\beta-\alpha)x} \\
		\end{dcases} 
	\end{aligned}
\end{equation*}
From post-examination, there are only two branches. Each branch will be diffeomorphic to a circle $S^1$ with the parametrization below:
\begin{equation*}
	\begin{dcases}
		x \in \mathbb{R} \cup \{\infty\} \\
		y = \dfrac{1}{x} \\
		z = x \\
		w = \dfrac{1}{x} 
	\end{dcases} \quad \mathrm{or} \quad \begin{dcases}
		x \in \mathbb{R} \cup \{\infty\} \\
		y = \dfrac{\alpha+\beta}{(\alpha-\beta)x} \\
		z = -x \\
		w = \dfrac{\beta+\alpha}{(\beta-\alpha)x} = -y
	\end{dcases}	
\end{equation*}
The first branch has a non-self-intersecting convex `car wiper' motion, while the second branch has a constantly self-intersecting `butterfly' shape.  

\section{Deltoid I: finite solution}

The condition on linkage lengths is
\begin{equation*}
	f_{22}=0, ~~ f_{20} \neq 0, ~~f_{02}=0 ~~ \Leftrightarrow ~~ \delta=\alpha, ~~ \gamma=\beta, ~~\beta \neq \alpha
\end{equation*}
which means:
\begin{equation*}
	\begin{aligned}
		\begin{cases}
			(\beta-\alpha)x^2-2\alpha xy+(\beta+\alpha) =0 \\
			z = \pm x \\
			(\alpha-\beta)x^2-2\beta xw+(\alpha+\beta) =0 
		\end{cases} 
		~~ \Rightarrow ~~  \begin{dcases}
			y = \dfrac{(\beta-\alpha)x^2+(\beta+\alpha)}{2\alpha x} \\
			z = \pm x \\
			w = \dfrac{(\alpha-\beta)x^2+(\alpha+\beta)}{2\beta x}
		\end{dcases} 
	\end{aligned}
\end{equation*}
From post-examination, there is only a single branch. Each branch will be diffeomorphic to a circle $S^1$ with the parametrization below:
\begin{equation*}
	\begin{dcases}
		x \in \mathbb{R} \cup \{\infty\} \\
		y = \dfrac{(\beta-\alpha)x^2+(\beta+\alpha)}{2\alpha x} \\
		z = x \\
		w = \dfrac{(\alpha-\beta)x^2+(\alpha+\beta)}{2\beta x}
	\end{dcases} 
\end{equation*}
This branch has a non-self-intersecting motion.

\section{Deltoid II: finite solution}

The condition on linkage lengths is 	
\begin{equation*}
	f_{22}=0, ~~ f_{20}=0 , ~~ f_{02} \neq 0 ~~ \Leftrightarrow ~~ \alpha=\beta, ~~ \delta=\gamma, ~~\gamma \neq \beta
\end{equation*}
which means:
\begin{equation*}
	\begin{aligned}
		\begin{cases}
			(\beta-\gamma)y^2-2\gamma xy +(\beta+\gamma) =0 \\
			\beta^2 (z^2 +1) = \gamma^2 (x^2+1) \\
			(\beta-\gamma)w^2-2\gamma xw+(\beta+\gamma) =0 
		\end{cases} 
		\Rightarrow ~~ \begin{dcases}
			y = \dfrac{\gamma x \pm \sqrt{\gamma^2(x^2+1)-\beta^2}}{\beta-\gamma} \\
			z = \pm \dfrac{\sqrt{\gamma^2(x^2+1)-\beta^2}}{\beta} \\
			w = \dfrac{\gamma x \pm \sqrt{\gamma^2(x^2+1)-\beta^2}}{\beta-\gamma}
		\end{dcases} 
	\end{aligned}
\end{equation*}
Only when $x^2 \ge \dfrac{\beta^2}{\gamma^2}-1$ is there a real solution. From post-examination, there are only two solutions:
\begin{equation} \label{eq: planar deltoid II}
	\begin{dcases}
		y = \dfrac{\gamma x + \sqrt{\gamma^2(x^2+1)-\beta^2}}{\beta-\gamma} \\
		z = - \dfrac{\sqrt{\gamma^2(x^2+1)-\beta^2}}{\beta} \\
		w = \dfrac{\gamma x + \sqrt{\gamma^2(x^2+1)-\beta^2}}{\beta-\gamma}
	\end{dcases} ~~ \mathrm{or} ~~
	\begin{dcases}
		y = \dfrac{\gamma x - \sqrt{\gamma^2(x^2+1)-\beta^2}}{\beta-\gamma} \\
		z =  \dfrac{\sqrt{\gamma^2(x^2+1)-\beta^2}}{\beta} \\
		w = \dfrac{\gamma x - \sqrt{\gamma^2(x^2+1)-\beta^2}}{\beta-\gamma}
	\end{dcases}
\end{equation}
where
\begin{equation*}
	\begin{aligned}
		\begin{dcases}
			x \in \left(-\infty, -\sqrt{\dfrac{\beta^2}{\gamma^2}-1} \right] \bigcup \left[\sqrt{\dfrac{\beta^2}{\gamma^2}-1}, +\infty \right) &  \mathrm{when} ~~ \beta>\gamma \\
			x \in \mathbb{R} &  \mathrm{when} ~~ \beta<\gamma
		\end{dcases}
	\end{aligned}
\end{equation*}

A more advanced expression is setting (referring to Subsection \ref{section: sign convention})
\begin{equation*}
	\begin{aligned}
		p_x=\sqrt{\dfrac{\beta^2}{\gamma^2}-1}, ~~
		\begin{dcases}
			x = \pm p_x \cosh s, ~~ s \in (0, +\infty) &  \mathrm{when} ~~ \beta>\gamma \\
			x = \pm i p_x \sinh s, ~~ s \in (0, +\infty) &  \mathrm{when} ~~ \beta<\gamma
		\end{dcases}
	\end{aligned}
\end{equation*}
We set $s \in (0, +\infty)$ in order for sign confirmation when calculating the square root. 

When $\beta>\gamma$,
\begin{equation*}
	\begin{aligned}
		&\begin{dcases}
			x = \sqrt{\dfrac{\beta^2}{\gamma^2}-1} \cosh s \\
			y = \sqrt{\dfrac{\beta+\gamma}{\beta-\gamma}} e^{s} \\
			z = -\sqrt{1-\dfrac{\gamma^2}{\beta^2}} \sinh s \\
			w= \sqrt{\dfrac{\beta+\gamma}{\beta-\gamma}} e^{s}
		\end{dcases} 
		~~ \mathrm{and} ~~ \begin{dcases}
			x = -\sqrt{\dfrac{\beta^2}{\gamma^2}-1} \cosh s \\
			y = -\sqrt{\dfrac{\beta+\gamma}{\beta-\gamma}} e^{s} \\
			z = \sqrt{1-\dfrac{\gamma^2}{\beta^2}} \sinh s \\
			w= -\sqrt{\dfrac{\beta+\gamma}{\beta-\gamma}} e^{s}
		\end{dcases} \\
	\end{aligned}
\end{equation*}
\begin{equation*}
	\begin{aligned}
		& ~~ \mathrm{or} ~~ \begin{dcases}
			x = \sqrt{\dfrac{\beta^2}{\gamma^2}-1} \cosh s \\
			y = \sqrt{\dfrac{\beta+\gamma}{\beta-\gamma}} e^{-s} \\
			z = \sqrt{1-\dfrac{\gamma^2}{\beta^2}} \sinh s \\
			w= \sqrt{\dfrac{\beta+\gamma}{\beta-\gamma}} e^{-s}
		\end{dcases} 
		~~ \mathrm{and} ~~ \begin{dcases}
			x = -\sqrt{\dfrac{\beta^2}{\gamma^2}-1} \cosh s \\
			y = -\sqrt{\dfrac{\beta+\gamma}{\beta-\gamma}} e^{-s} \\
			z = -\sqrt{1-\dfrac{\gamma^2}{\beta^2}} \sinh s \\
			w= -\sqrt{\dfrac{\beta+\gamma}{\beta-\gamma}} e^{-s}
		\end{dcases} 
	\end{aligned} , ~~ s \in (0, +\infty)
\end{equation*}
When $\beta < \gamma$
\begin{equation*}
	\begin{aligned}
		& \begin{dcases}
			x = -\sqrt{1-\dfrac{\beta^2}{\gamma^2}} \sinh s \\
			y = \sqrt{\dfrac{\gamma+\beta}{\gamma-\beta}} e^{s} \\
			z = \sqrt{\dfrac{\gamma^2}{\beta^2}-1} \cosh s \\
			w= \sqrt{\dfrac{\gamma+\beta}{\gamma-\beta}} e^{s}
		\end{dcases} 
		~~ \mathrm{and} ~~ \begin{dcases}
			x = \sqrt{1-\dfrac{\beta^2}{\gamma^2}} \sinh s \\
			y = -\sqrt{\dfrac{\gamma+\beta}{\gamma-\beta}} e^{s} \\
			z = -\sqrt{\dfrac{\gamma^2}{\beta^2}-1} \cosh s \\
			w= -\sqrt{\dfrac{\gamma+\beta}{\gamma-\beta}} e^{s}
		\end{dcases}
	\end{aligned}
\end{equation*}
\begin{equation*}
	\begin{aligned}
		& ~~ \mathrm{or} ~~ \begin{dcases}
			x = \sqrt{1-\dfrac{\beta^2}{\gamma^2}} \sinh s \\
			y = \sqrt{\dfrac{\gamma+\beta}{\gamma-\beta}} e^{-s} \\
			z = \sqrt{\dfrac{\gamma^2}{\beta^2}-1} \cosh s \\
			w= \sqrt{\dfrac{\gamma+\beta}{\gamma-\beta}} e^{-s}
		\end{dcases}
		~~ \mathrm{and} ~~ \begin{dcases}
			x = -\sqrt{1-\dfrac{\beta^2}{\gamma^2}} \sinh s \\
			y = -\sqrt{\dfrac{\gamma+\beta}{\gamma-\beta}} e^{-s} \\
			z = -\sqrt{\dfrac{\gamma^2}{\beta^2}-1} \cosh s \\
			w= -\sqrt{\dfrac{\gamma+\beta}{\gamma-\beta}} e^{-s}
		\end{dcases} 
	\end{aligned} , ~~ s \in (0, +\infty)
\end{equation*}

Actually we could write the expressions in a more compact form: 

When $\beta>\gamma$:
\begin{equation*}
	\begin{dcases}
		x = \sqrt{\dfrac{\beta^2}{\gamma^2}-1} \cosh s \\
		y = \sqrt{\dfrac{\beta+\gamma}{\beta-\gamma}} e^{-s} \\
		z = \sqrt{1-\dfrac{\gamma^2}{\beta^2}} \sinh s \\
		w= \sqrt{\dfrac{\beta+\gamma}{\beta-\gamma}} e^{-s}
	\end{dcases}, ~~ \mathrm{and} ~~
	\begin{dcases}
		x = -\sqrt{\dfrac{\beta^2}{\gamma^2}-1} \cosh s \\
		y = -\sqrt{\dfrac{\beta+\gamma}{\beta-\gamma}} 	e^{-s} \\
		z = -\sqrt{1-\dfrac{\gamma^2}{\beta^2}} \sinh s \\
		w= -\sqrt{\dfrac{\beta+\gamma}{\beta-\gamma}} e^{-s}
	\end{dcases} , ~~ s \in (-\infty, 0) \cup (0, +\infty)
\end{equation*}
When $\gamma<\beta$:
\begin{equation*}
	\begin{dcases}
		x = \sqrt{1-\dfrac{\beta^2}{\gamma^2}} \sinh s \\
		y = \sqrt{\dfrac{\gamma+\beta}{\gamma-\beta}} e^{-s} \\
		z = \sqrt{\dfrac{\gamma^2}{\beta^2}-1} \cosh s \\
		w= \sqrt{\dfrac{\gamma+\beta}{\gamma-\beta}} e^{-s}
	\end{dcases} ~~\mathrm{and} ~~\begin{dcases}
		x = -\sqrt{1-\dfrac{\beta^2}{\gamma^2}} \sinh s \\
		y = -\sqrt{\dfrac{\gamma+\beta}{\gamma-\beta}} e^{-s} \\
		z = -\sqrt{\dfrac{\gamma^2}{\beta^2}-1} \cosh s \\
		w= -\sqrt{\dfrac{\gamma+\beta}{\gamma-\beta}} e^{-s}
	\end{dcases}, ~~ s \in (-\infty, 0) \cup (0, +\infty) 
\end{equation*}

Further, we want to do an \textbf{analytical continuation} to the above real solutions towards the one-point compactification of the complex field $\mathbb{C} \cup \{\infty\}$, which, as we will see later, is a more helpful and unified expression for future symbolic calculations. Let  
\begin{equation*}
	\begin{aligned}
		x=p_x \cos t  
	\end{aligned}
\end{equation*}
When $\beta>\gamma$, 
\begin{equation*}
	\begin{dcases}
		x = \sqrt{\dfrac{\beta^2}{\gamma^2}-1} \cos t \\
		y = \sqrt{\dfrac{\beta+\gamma}{\beta-\gamma}} e^{it} \\
		z = -i\sqrt{1-\dfrac{\gamma^2}{\beta^2}} \sin t = i\sqrt{1-\dfrac{\gamma^2}{\beta^2}} \cos \left(t+\dfrac{\pi}{2}\right) \\
		w= \sqrt{\dfrac{\beta+\gamma}{\beta-\gamma}} e^{it}
	\end{dcases} 	 
\end{equation*}
When $\beta < \gamma$,
\begin{equation*}
	\begin{dcases}
		x = i \sqrt{1-\dfrac{\beta^2}{\gamma^2}} \cos t \\
		y = \sqrt{\dfrac{\gamma+\beta}{\gamma-\beta}} e^{it} \\
		z = \sqrt{\dfrac{\gamma^2}{\beta^2}-1} \sin t = \sqrt{\dfrac{\gamma^2}{\beta^2}-1} \cos \left(t-\dfrac{\pi}{2} \right) \\
		w= \sqrt{\dfrac{\gamma+\beta}{\gamma-\beta}} e^{it}
	\end{dcases} 	 
\end{equation*}
By carefully choosing $t$ for real configurations, with the sign convention in Equation \eqref{section: sign convention}, it is possible to write the two branches in a unified form. Each branch will be diffeomorphic to $\mathbb{R}$ with the parametrization below: 
\begin{equation} \label{eq: planar deltoid 2 complex}
	\begin{dcases}
		x = p_x \cos t \\
		y = \sqrt{\dfrac{\beta+\gamma}{\beta-\gamma}} e^{it} \\
		z = p_z \cos \left(t+ \dfrac{\pi}{2}\right) \\
		w= \sqrt{\dfrac{\beta+\gamma}{\beta-\gamma}} e^{it}
	\end{dcases} 	 
\end{equation}
\bgroup
\def\arraystretch{2}
\begin{center}
	\begin{tabular}{|c|c|c|}
		\hline
		\makecell{Choice of $t$ for \\ real configurations} &  Branch 1 & Branch 2 \\ 
		\hline
		\makecell{$p_x \in \mathbb{R}^+$ \\ $p_z \in i\mathbb{R}^+$} & \makecell {$is$, $s \in (0, +\infty) $ \\ $\pi + is$, $s \in (0, +\infty)$ \\ when $xz \neq 0$, $xz>0$ \\ this branch is continuous at $x = \infty$ \\ and `snaps' at $x = 0$} & \makecell {$is$, $s \in (-\infty, 0)$ \\ $\pi + is$, $s \in (-\infty, 0)$ \\ when $xz \neq 0$, $xz>0$ \\ this branch is continuous at $x = \infty$ \\ and `snaps' at $x = 0$} \\
		\hline
		\makecell{$p_x \in i\mathbb{R}^+$ \\ $p_z \in \mathbb{R}^+$} & \makecell {$3\pi/2 + is$, $s \in (0, +\infty) $ \\ $\pi/2 + is$, $s \in (0, +\infty) $ \\ when $xz \neq 0$, $xz>0$ \\ this branch is continuous at $x = \infty$ \\ and `snaps' at $x = 0$} & \makecell {$\pi/2 + is$, $s \in (-\infty, 0)$ \\ $3\pi/2 + is$, $s \in (-\infty, 0)$ \\ when $xz \neq 0$, $xz<0$ \\ this branch is continuous at $x = \infty$ \\ and `snaps' at $x = 0$} \\  
		\hline      
	\end{tabular}
\end{center}
\egroup

There is no self-intersection for the two branches of flexes. Note that it is also possible to set $x = p_x \sin t$.

\section{Conic I: finite solution} \label{section: planar conic I}

The condition on linkage lengths is 	
\begin{equation*} 
	f_{22}=0, ~~ f_{20} \neq 0 , ~~ f_{02} \neq 0  ~~ \Leftrightarrow ~~ \begin{dcases}
		\alpha - \beta + \gamma - \delta = 0 \\
		\alpha - \beta - \gamma + \delta \neq 0 \\
		\alpha + \beta - \gamma - \delta \neq 0
	\end{dcases}
\end{equation*} 
which implies $\sigma=\alpha+\gamma=\beta+\delta$, and:
\begin{equation*} 
	\begin{aligned}
		\begin{cases}
			\alpha(\beta-\gamma)y^2-2\alpha\gamma xy+\gamma(\beta-\alpha)x^2+\beta(\alpha+\gamma) =0 \\
			\alpha \beta (z^2+1)=\gamma \delta (x^2+1) \\
			\beta(\alpha-\delta)w^2-2\beta\delta xw+\delta(\alpha-\beta)x^2+\alpha(\beta+\delta)=0
		\end{cases}
	\end{aligned}
\end{equation*}
First (referring to Subsection \ref{section: sign convention}),
\begin{equation*}
	\begin{aligned}
		\alpha \beta (z^2+1)=\gamma \delta (x^2+1) & ~~ \Leftrightarrow ~~ \gamma \delta x^2 - \alpha \beta z^2 = \alpha \beta - \gamma \delta
		\\
		& ~~ \Leftrightarrow ~~ \dfrac{\gamma \delta x^2}{\alpha \beta-\gamma \delta} +  \dfrac{\alpha \beta z^2}{\gamma \delta -\alpha \beta} = 1  \\
		& ~~ \Leftrightarrow ~~ \dfrac{x^2}{\dfrac{\alpha \beta}{\gamma \delta}-1} + \dfrac{z^2}{\dfrac{\gamma \delta}{\alpha \beta}-1} = 1 \\
		&  ~~ \Leftrightarrow ~~ \dfrac{x^2}{p_x^2} + \dfrac{z^2}{p_z^2} = 1
	\end{aligned}
\end{equation*}
This equation is always hyperbolic. We will start with parametrization in the real field. 

It is worth mentioning that the magnitude of $\alpha, ~ \beta, ~\gamma, ~\delta$ plays an important role in determining the signs of expressions below, as shown in the following proposition.
\begin{prop}
	When $\alpha + \gamma = \beta + \delta$,
	\begin{enumerate} [label={[\arabic*]}]
		\item $\beta = \mathrm{max}$ implies $\alpha \beta > \gamma \delta, ~ \beta \gamma > \delta \alpha, ~ \alpha \gamma > \beta \delta$; any two of $\alpha \beta > \gamma \delta, ~ \beta \gamma > \delta \alpha, ~ \alpha \gamma > \beta \delta$ implies $\beta = \mathrm{max}$.
		\begin{equation*}
			\alpha \beta > \gamma \delta, ~ \beta \gamma > \delta \alpha \Leftrightarrow p_x \in \mathbb{R}^+, ~p_y \in \mathbb{R}^+, ~p_z \in i\mathbb{R}^+, p_w \in i\mathbb{R}^+
		\end{equation*}
		\item $\alpha = \mathrm{max}$ implies $\alpha \beta > \gamma \delta, ~ \beta \gamma < \delta \alpha, ~ \alpha \gamma < \beta \delta$; any two of $\alpha \beta > \gamma \delta, ~ \beta \gamma < \delta \alpha, ~ \alpha \gamma < \beta \delta$ implies $\alpha = \mathrm{max}$.
		\begin{equation*}
			\alpha \beta > \gamma \delta, ~ \beta \gamma < \delta \alpha \Leftrightarrow p_x \in \mathbb{R}^+, ~p_y \in i\mathbb{R}^+, ~p_z \in i\mathbb{R}^+, p_w \in \mathbb{R}^+
		\end{equation*}
		\item $\delta = \mathrm{max}$ implies $\alpha \beta < \gamma \delta, ~ \beta \gamma < \delta \alpha, ~ \alpha \gamma > \beta \delta$; any two of $\alpha \beta < \gamma \delta, ~ \beta \gamma > \delta \alpha, ~ \alpha \gamma > \beta \delta$ implies $\delta = \mathrm{max}$.
		\begin{equation*}
			\alpha \beta < \gamma \delta, ~ \beta \gamma < \delta \alpha \Leftrightarrow p_x \in i\mathbb{R}^+, ~p_y \in i\mathbb{R}^+, ~p_z \in \mathbb{R}^+, p_w \in \mathbb{R}^+
		\end{equation*}
		\item $\gamma = \mathrm{max}$ implies $\alpha \beta < \gamma \delta, ~\beta \gamma > \delta \alpha , ~ \alpha \gamma < \beta \delta$; any two of $\alpha \beta < \gamma \delta, ~ \beta \gamma > \delta \alpha, ~ \alpha \gamma < \beta \delta$ implies $\beta = \mathrm{max}$.
		\begin{equation*}
			\alpha \beta < \gamma \delta, ~ \beta \gamma > \delta \alpha \Leftrightarrow p_x \in i\mathbb{R}^+, ~p_y \in \mathbb{R}^+, ~p_z \in \mathbb{R}^+, p_w \in i\mathbb{R}^+
		\end{equation*}
	\end{enumerate}
\end{prop}

Let us take one of the cases above as an example to show the derivation. When $\beta = \mathrm{max}$,
\begin{equation*}
	x = \pm \sqrt{\dfrac{\alpha \beta}{\gamma \delta}-1} \cosh s, \quad z = \pm \sqrt{1-\dfrac{\gamma \delta}{\alpha \beta}} \sinh s, ~~ s \in (0, +\infty)
\end{equation*}
Substitute $x$ into the quadratic equation with respect to $y$ and $w$, we have
\begin{equation*}
	\begin{gathered}
		\alpha(\beta-\gamma)y^2 \mp 2\alpha\gamma y \sqrt{\dfrac{\alpha \beta}{\gamma \delta}-1} \cosh s+ \dfrac{(\beta-\alpha)(\alpha \beta -\gamma \delta) \cosh^2 s}{\delta}+\beta(\alpha+\gamma) =0 \\
		\beta(\alpha-\delta)y^2 \mp 2\beta\delta y \sqrt{\dfrac{\alpha \beta}{\gamma \delta}-1} \cosh s+ \dfrac{(\alpha-\beta(\alpha \beta -\gamma \delta) \cosh^2 s}{\beta}+\delta(\beta+\delta) =0
	\end{gathered}
\end{equation*}
and the solutions are 
\begin{equation*}
	\begin{gathered}
		y= \sqrt{\dfrac{\beta+\delta}{\beta-\gamma}} \left( \pm \sqrt{\dfrac{\gamma}{\delta}}  \cosh s \pm \sqrt{\dfrac{\beta}{\alpha}} \sinh s \right) \\
		w= \sqrt{\dfrac{\alpha+\gamma}{\alpha-\delta}} \left( \pm \sqrt{\dfrac{\delta}{\gamma}}  \cosh s \pm \sqrt{\dfrac{\alpha}{\beta}} \sinh s \right) 
	\end{gathered}
\end{equation*}
Take one sign choice of $y$ and $w$ as an example:
\begin{equation*}
	\begin{aligned}
		y_1 & = \sqrt{\dfrac{\beta+\delta}{\beta-\gamma}} \left( \sqrt{\dfrac{\gamma}{\delta}}  \cosh s + \sqrt{\dfrac{\beta}{\alpha}} \sinh s \right) \\
		& = \sqrt{\dfrac{\beta+\delta}{\beta-\gamma}\left(\dfrac{\gamma}{\delta}-\dfrac{\beta}{\alpha}\right)} \left(\sqrt{\dfrac{\alpha \gamma}{ \alpha \gamma - \beta \delta}} \cosh s + \sqrt{\dfrac{\beta \delta}{ \alpha \gamma - \beta \delta}} \sinh s \right)  \\
		& = \sqrt{\dfrac{\beta \gamma}{\delta \alpha }-1} \left(\sqrt{\dfrac{\alpha \gamma}{ \alpha \gamma - \beta \delta}} \cosh s + \sqrt{\dfrac{\beta \delta}{ \alpha \gamma - \beta \delta}} \sinh s \right) \\
		& = \sqrt{\dfrac{\beta \gamma}{\delta \alpha }-1} \cosh (s + \theta_1') \\
	\end{aligned}
\end{equation*}
\begin{equation*}
	\begin{aligned}
		w_1 & = \sqrt{\dfrac{\alpha+\gamma}{\alpha-\delta}} \left( \sqrt{\dfrac{\delta}{\gamma}}  \cosh s + \sqrt{\dfrac{\alpha}{\beta}} \sinh s \right) \\
		& = \sqrt{\dfrac{\alpha+\gamma}{\alpha-\delta}\left(\dfrac{\alpha}{\beta}-\dfrac{\delta}{\gamma}\right)} \left(\sqrt{\dfrac{\beta \delta}{ \alpha \gamma - \beta \delta}} \cosh s + \sqrt{\dfrac{\alpha \gamma}{ \alpha \gamma - \beta \delta}} \sinh s \right)  \\
		& = \sqrt{1-\dfrac{\delta \alpha}{\beta \gamma}} \left(\sqrt{\dfrac{\beta \delta}{ \alpha \gamma - \beta \delta}} \cosh s + \sqrt{\dfrac{\alpha \gamma}{ \alpha \gamma - \beta \delta}} \sinh s \right) \\
		& = \sqrt{1-\dfrac{\delta \alpha}{\beta \gamma}} \sinh (s + \theta_1')
	\end{aligned}
\end{equation*}
\begin{equation*}
	\tanh \theta_1' = \sqrt{\dfrac{\beta \delta}{\alpha \gamma}} ~~ \Rightarrow ~~ \theta_1' = \dfrac{1}{2} \ln \left( \dfrac{\sqrt{\beta \delta} + \sqrt{\alpha \gamma}}{\sqrt{\beta \delta} - \sqrt{\alpha \gamma}}\right)
\end{equation*}
We use $\theta_1'$ here since the phase shift $\theta_1$ will be defined later. After post-examination on sign choices, the solutions are:
\begin{equation*}
	\begin{aligned}
		&\begin{dcases}
			x = \sqrt{\dfrac{\alpha \beta}{\gamma \delta}-1} \cosh s \\
			y = \sqrt{\dfrac{\beta \gamma}{\delta \alpha }-1} \cosh (s + \theta_1') \\
			z = -\sqrt{1-\dfrac{\gamma \delta}{\alpha \beta}} \sinh s \\
			w= \sqrt{1-\dfrac{\delta \alpha}{\beta \gamma}} \sinh (s + \theta_1')
		\end{dcases} 
		~~ \mathrm{and} ~~ \begin{dcases}
			x = -\sqrt{\dfrac{\alpha \beta}{\gamma \delta}-1} \cosh s \\
			y = -\sqrt{\dfrac{\beta \gamma}{\delta \alpha }-1} \cosh (s + \theta_1') \\
			z = \sqrt{1-\dfrac{\gamma \delta}{\alpha \beta}} \sinh s \\
			w= -\sqrt{1-\dfrac{\delta \alpha}{\beta \gamma}} \sinh (s + \theta_1')
		\end{dcases}
	\end{aligned}
\end{equation*}
\begin{equation*}
	\begin{aligned}
		\mathrm{and} ~~ \begin{dcases}
			x = \sqrt{\dfrac{\alpha \beta}{\gamma \delta}-1} \cosh s \\
			y = \sqrt{\dfrac{\beta \gamma}{\delta \alpha }-1} \cosh (-s + \theta_1') \\
			z = \sqrt{1-\dfrac{\gamma \delta}{\alpha \beta}} \sinh s \\
			w= \sqrt{1-\dfrac{\delta \alpha}{\beta \gamma}} \sinh (-s + \theta_1')
		\end{dcases} 
		~~ \mathrm{and} ~~ \begin{dcases}
			x = -\sqrt{\dfrac{\alpha \beta}{\gamma \delta}-1} \cosh s \\
			y = -\sqrt{\dfrac{\beta \gamma}{\delta \alpha }-1} \cosh (-s + \theta_1') \\
			z = -\sqrt{1-\dfrac{\gamma \delta}{\alpha \beta}} \sinh s \\
			w= -\sqrt{1-\dfrac{\delta \alpha}{\beta \gamma}} \sinh (-s + \theta_1')
		\end{dcases} 
	\end{aligned}, ~~s \in (0, +\infty)
\end{equation*}

By extending $s \in (0, +\infty)$ to $s \in (-\infty, 0) \cup (0, +\infty)$, the flexes above could be rewritten as:

\begin{equation*}
	\begin{dcases}
		x = \sqrt{\dfrac{\alpha \beta}{\gamma \delta}-1} \cosh s \\
		y = \sqrt{\dfrac{\beta \gamma}{\delta \alpha }-1} \cosh (s - \theta_1') \\
		z = \sqrt{1-\dfrac{\gamma \delta}{\alpha \beta}} \sinh s \\
		w= -\sqrt{1-\dfrac{\delta \alpha}{\beta \gamma}} \sinh (s - \theta_1')
	\end{dcases} 
	~~ \mathrm{or} ~~ \begin{dcases}
		x = -\sqrt{\dfrac{\alpha \beta}{\gamma \delta}-1} \cosh s \\
		y = -\sqrt{\dfrac{\beta \gamma}{\delta \alpha }-1} \cosh (s - \theta_1') \\
		z = -\sqrt{1-\dfrac{\gamma \delta}{\alpha \beta}} \sinh s \\
		w= \sqrt{1-\dfrac{\delta \alpha}{\beta \gamma}} \sinh (s - \theta_1')
	\end{dcases} 
\end{equation*}

Now we could consider the analytical continuation, \begin{equation*}
	\begin{aligned}
		x=p_x \cos t  
	\end{aligned}
\end{equation*}
$t$ can be periodically extended to the one-point compactification of the complex field $\mathbb{C} \cup \{\infty\}$.
\begin{equation*}
	\begin{gathered}
		\begin{dcases}
			x = \sqrt{\dfrac{\alpha \beta}{\gamma \delta}-1} \cos t \\
			y = \sqrt{\dfrac{\beta \gamma}{\delta \alpha }-1} \cos \left( t - i \theta_1' \right) \\
			z = -i\sqrt{1 - \dfrac{\gamma \delta}{\alpha \beta}} \sin t = \sqrt{\dfrac{\gamma \delta}{\alpha \beta} - 1} \cos \left(t + \dfrac{\pi}{2}\right) \\
			w= i\sqrt{1-\dfrac{\delta \alpha}{\beta \gamma}}\sin (t - i \theta_1')  = \sqrt{\dfrac{\delta \alpha}{\beta \gamma}-1} \cos \left(t - i\theta_1' - \dfrac{\pi}{2} \right)
		\end{dcases} \\
		t = is ~~\mathrm{or}~~ \pi+is \mathrm{~for~real~configurations}\\
		\theta_1' = \dfrac{1}{2} \ln \left(\dfrac{\sqrt{\alpha \gamma}+\sqrt{\beta \delta}}{\sqrt{\alpha \gamma}-\sqrt{\beta \delta}}\right)
	\end{gathered}
\end{equation*}

Similar to Deltoid II, after listing all the cases, our next step is apply analytical continuation and write the final expressions. Let
\begin{equation*}
	\begin{gathered}
		\tan \theta_1 = i \sqrt{\dfrac{\beta \delta}{\alpha \gamma}} \\
		\tan \theta_2 = i \sqrt{\dfrac{\alpha \gamma}{\beta \delta}}
	\end{gathered}
\end{equation*}
and we could simplify the expression by choosing proper phase shift of $\theta_1$ and $\theta_2$. It is possible to write the two branches in a unified form. Each branch will be diffeomorphic to $\mathbb{R}$ with the parametrization below:
\begin{equation}
	\begin{dcases}
		x = p_x \cos t \\
		y = p_y \cos \left( t - \theta_1 \right) \\
		z = p_z \cos \left(t + \dfrac{\pi}{2}\right) \\
		w = p_w \cos \left(t - \theta_2 \right)
	\end{dcases}
\end{equation}
We could clearly see that in the complex field, $x, ~y, ~ z, ~w$ are trigonometric functions with specific amplitudes and phase shifts. 
\bgroup
\def\arraystretch{2}
\begin{center}
	\begin{tabular}{|c|c|c|}
		\hline
		\makecell{Choice of $t$ for \\ real configurations} &  Branch 1 & Branch 2 \\ 
		\hline
		\makecell{$p_x \in \mathbb{R}^+$ \\ $p_z \in i\mathbb{R}^+$} & \makecell {$is$, $s \in (0, +\infty)$ \\ $\pi + is$, $s \in (0, +\infty)$ \\ when $xz \neq 0$, $xz>0$ \\ this branch is continuous at $x = \infty$ \\ and `snaps' at $x = 0$} & \makecell {$is$, $s \in (-\infty, 0)$ \\ $\pi + is$, $s \in (-\infty, 0)$ \\ when $xz \neq 0$, $xz<0$ \\ this branch is continuous at $x = \infty$ \\ and `snaps' at $x = 0$} \\
		\hline
		\makecell{$p_x \in i\mathbb{R}^+$ \\ $p_z \in \mathbb{R}^+$} & \makecell {$3\pi/2 + is$, $s \in (0, +\infty) $ \\ $\pi/2 + is$, $s \in (0, +\infty)$ \\ when $xz \neq 0$, $xz>0$ \\ this branch is continuous at $x = \infty$ \\ and `snaps' at $x = 0$} & \makecell {$\pi/2 + is$, $s \in (-\infty, 0)$ \\ $3\pi/2 + is$, $s \in (-\infty, 0)$ \\ when $xz \neq 0$, $xz<0$ \\ this branch is continuous at $x = \infty$ \\ and `snaps' at $x = 0$} \\  
		\hline      
	\end{tabular}
\end{center}
Next, for the phase shift:
\begin{center}
	\begin{tabular}{|c|c|c|}
		\hline
		&  $\theta_1$ & $\theta_2$ \\
		\hline
		$p_x \in \mathbb{R}^+, ~p_y \in \mathbb{R}^+ \Leftrightarrow \beta = \mathrm{max}$ & $\dfrac{i}{2}\ln \left|\dfrac{\sqrt{\alpha \gamma}+\sqrt{\beta \delta}}{\sqrt{\alpha \gamma}-\sqrt{\beta \delta}}\right|$ & $\dfrac{\pi}{2} + \theta_1$ \\
		\hline
		$p_y \in \mathbb{R}^+, ~p_z \in \mathbb{R}^+ \Leftrightarrow \gamma = \mathrm{max}$ & $-\dfrac{\pi}{2} + \dfrac{i}{2}\ln \left|\dfrac{\sqrt{\alpha \gamma}+\sqrt{\beta \delta}}{\sqrt{\alpha \gamma}-\sqrt{\beta \delta}}\right|$ & $\dfrac{\pi}{2} + \theta_1$\\ 
		\hline 
		$p_z \in \mathbb{R}^+, ~p_w \in \mathbb{R}^+ \Leftrightarrow \delta = \mathrm{max}$ & $ \dfrac{i}{2}\ln \left|\dfrac{\sqrt{\alpha \gamma}+\sqrt{\beta \delta}}{\sqrt{\alpha \gamma}-\sqrt{\beta \delta}}\right|$ & $-\dfrac{\pi}{2} + \theta_1$ \\
		\hline
		$p_w \in \mathbb{R}^+, ~p_x \in \mathbb{R}^+ \Leftrightarrow \alpha = \mathrm{max}$ & $ \dfrac{\pi}{2} + \dfrac{i}{2}\ln \left|\dfrac{\sqrt{\alpha \gamma}+\sqrt{\beta \delta}}{\sqrt{\alpha \gamma}-\sqrt{\beta \delta}}\right|$ & $-\dfrac{\pi}{2} + \theta_1$ \\
		\hline      
	\end{tabular}
\end{center}
\egroup
There is no self-intersection for the above two branches.

\section{Conic II: finite solution}

The condition on linkage lengths is 
\begin{equation*} 
	f_{22} \neq 0, ~~ f_{20} = 0 , ~~ f_{02} \neq 0 ~~ \Leftrightarrow ~~ \begin{dcases}
		\alpha - \beta + \gamma - \delta \neq 0 \\
		\alpha - \beta - \gamma + \delta = 0 \\
		\alpha + \beta - \gamma - \delta \neq 0
	\end{dcases}
\end{equation*}
\begin{equation*} 
	\begin{aligned}
		& \begin{cases}
			\gamma(\alpha-\beta)x^2y^2-2\alpha\gamma xy+\beta(\alpha-\gamma)y^2+\alpha(\beta+\gamma) =0 \\
			\alpha \beta (z^2+1)=\gamma \delta (x^2+1) \\
			\delta(\beta-\gamma)x^2w^2-2\beta\delta xw+\gamma(\beta-\alpha)w^2+\beta(\alpha+\gamma)=0
		\end{cases} \\
		= & \begin{cases}
			\alpha(\beta+\gamma)y^{-2} - 2\alpha\gamma y^{-1}x + \gamma(\alpha-\beta)x^2 + \beta(\alpha-\gamma) =0 \\
			\alpha \beta (z^2+1)=\gamma \delta (x^2+1) \\
			\beta(\alpha+\delta)w^{-2} - 2\beta\delta w^{-1}x + \delta(\beta-\alpha)x^2 + \alpha(\beta-\delta) =0
		\end{cases} \\
	\end{aligned}
\end{equation*}
The substitution below allows us to transfer the result for Conic I to Conic II:
\begin{equation}
	(\alpha, ~\beta, ~\gamma, ~\delta, ~x, ~y, ~z, ~w) \rightarrow (\alpha, ~\beta, ~-\gamma, ~-\delta, ~x, ~-y^{-1}, ~-z, ~-w^{-1}) \\
\end{equation}
Now:
\begin{equation*}
	\sigma = \dfrac{\alpha + \beta - \gamma - \delta}{2}
\end{equation*}
and we could directly write the result with a careful checking on the sign change:
\begin{equation}
	\begin{dcases}
		x = p_x \cos t \\
		y^{-1} = \mathrm{sign}(\sigma) p_y \cos \left( t - \theta_1 \right) \\
		z = p_z \cos \left(t - \dfrac{\pi}{2}\right) \\
		w^{-1} = \mathrm{sign}(\sigma) p_w \cos \left(t - \theta_2 \right)
	\end{dcases}
\end{equation}
\bgroup
\def\arraystretch{2}
\begin{center}
	\begin{tabular}{|c|c|c|}
		\hline
		\makecell{Choice of $t$ for \\ real configurations} &  Branch 1 & Branch 2 \\ 
		\hline
		\makecell{$p_x \in \mathbb{R}^+$ \\ $p_z \in i\mathbb{R}^+$} & \makecell {$is$, $s \in (-\infty, 0)$ \\ $\pi + is$, $s \in (-\infty, 0)$ \\ when $xz \neq 0$, $xz>0$ \\ this branch is continuous at $x = \infty$ \\ and `snaps' at $x = 0$} & \makecell {$is$, $s \in (0, +\infty)$ \\ $\pi + is$, $s \in (0, +\infty)$ \\ when $xz \neq 0$, $xz<0$ \\ this branch is continuous at $x = \infty$ \\ and `snaps' at $x = 0$} \\
		\hline
		\makecell{$p_x \in i\mathbb{R}^+$ \\ $p_z \in \mathbb{R}^+$} & \makecell {$\pi/2 + is$, $s \in (-\infty, 0)$ \\ $3\pi/2 + is$, $s \in (-\infty, 0)$ \\ when $xz \neq 0$, $xz>0$ \\ this branch is continuous at $x = \infty$ \\ and `snaps' at $x = 0$} & \makecell {$3\pi/2 + is$, $s \in (0, +\infty)$ \\ $\pi/2 + is$, $s \in (0, +\infty)$ \\ when $xz \neq 0$, $xz>0$ \\ this branch is continuous at $x = \infty$ \\ and `snaps' at $x = 0$} \\  
		\hline      
	\end{tabular}
\end{center}
Next, for the phase shift:
\begin{center}
	\begin{tabular}{|c|c|c|}
		\hline
		&  $\theta_1$ & $\theta_2$ \\
		\hline
		$p_x \in \mathbb{R}^+, ~p_y \in \mathbb{R}^+ \Leftrightarrow \beta = \mathrm{max}$ & $\dfrac{i}{2}\ln \left|\dfrac{\sqrt{\alpha \gamma}+\sqrt{\beta \delta}}{\sqrt{\alpha \gamma}-\sqrt{\beta \delta}}\right|$ & $\dfrac{\pi}{2} + \theta_1$ \\
		\hline
		$p_y \in \mathbb{R}^+, ~p_z \in \mathbb{R}^+ \Leftrightarrow \gamma = \mathrm{max}$ & $-\dfrac{\pi}{2} + \dfrac{i}{2}\ln \left|\dfrac{\sqrt{\alpha \gamma}+\sqrt{\beta \delta}}{\sqrt{\alpha \gamma}-\sqrt{\beta \delta}}\right|$ & $\dfrac{\pi}{2} + \theta_1$\\ 
		\hline 
		$p_z \in \mathbb{R}^+, ~p_w \in \mathbb{R}^+ \Leftrightarrow \delta = \mathrm{max}$ & $ \dfrac{i}{2}\ln \left|\dfrac{\sqrt{\alpha \gamma}+\sqrt{\beta \delta}}{\sqrt{\alpha \gamma}-\sqrt{\beta \delta}}\right|$ & $-\dfrac{\pi}{2} + \theta_1$ \\
		\hline
		$p_w \in \mathbb{R}^+, ~p_x \in \mathbb{R}^+ \Leftrightarrow \alpha = \mathrm{max}$ & $ \dfrac{\pi}{2} + \dfrac{i}{2}\ln \left|\dfrac{\sqrt{\alpha \gamma}+\sqrt{\beta \delta}}{\sqrt{\alpha \gamma}-\sqrt{\beta \delta}}\right|$ & $-\dfrac{\pi}{2} + \theta_1$ \\
		\hline      
	\end{tabular}
\end{center}
\egroup

\section{Conic III: finite solution}

The condition on linkage lengths is 
\begin{equation*} 
	f_{22} \neq 0, ~~ f_{20} \neq 0 , ~~ f_{02} = 0 ~~ \Leftrightarrow ~~ \begin{dcases}
		\alpha - \beta + \gamma - \delta \neq 0 \\
		\alpha - \beta - \gamma + \delta \neq 0 \\
		\alpha + \beta - \gamma - \delta = 0
	\end{dcases}
\end{equation*}
which implies $\sigma  = \alpha + \beta = \gamma + \delta$, and
\begin{equation*} 
	\begin{aligned}
		& \begin{cases}
			\alpha(\gamma-\beta)x^2y^2+\beta(\gamma-\alpha)x^2-2\alpha\gamma xy+\gamma(\alpha+\beta) =0 \\
			\alpha \beta (z^{-2}+1)=\gamma \delta (x^{-2}+1) \\
			\beta(\delta-\alpha)x^2w^2+\alpha(\delta-\beta)x^2-2\beta\delta xw+\delta(\beta+\alpha) =0 \\
		\end{cases} \\
		= & \begin{cases}
			\alpha(\gamma-\beta)y^2-2\alpha\gamma x^{-1}y+\gamma(\alpha+\beta)x^{-2}+\beta(\gamma-\alpha) =0 \\
			\alpha \beta (z^{-2}+1)=\gamma \delta (x^{-2}+1) \\
			\beta(\delta-\alpha)w^2-2\beta\delta x^{-1}w+\delta(\beta+\alpha)x^{-2}+\alpha(\delta-\beta) =0 \\
		\end{cases} \\
	\end{aligned}
\end{equation*}
The substitution below allows us to transfer the result for Conic I to Conic III:
\begin{equation}
	(\alpha, ~\beta, ~\gamma, ~\delta, ~x, ~y, ~z, ~w) \rightarrow (-\alpha, ~\beta, ~\gamma, ~-\delta, ~-x^{-1}, ~y, ~-z^{-1}, ~-w) \\
\end{equation}
Now:
\begin{equation*}
	\sigma = \dfrac{- \alpha + \beta + \gamma - \delta}{2}
\end{equation*}
and we could directly write the result with a little careful checking on the derivation for Conic I:
\begin{equation}
	\begin{dcases}
		x^{-1} = -p_x \cos t \\
		y = -\mathrm{sign}(\sigma) p_y \cos \left( t - \theta_1 \right) \\
		z^{-1} = p_z \cos \left(t - \dfrac{\pi}{2}\right) \\
		w = \mathrm{sign}(\sigma) p_w \cos \left(t - \theta_2 \right)
	\end{dcases}
\end{equation}
\bgroup
\def\arraystretch{2}
\begin{center}
	\begin{tabular}{|c|c|c|}
		\hline
		\makecell{Choice of $t$ for \\ real configurations} &  Branch 1 & Branch 2 \\ 
		\hline
		\makecell{$p_x \in \mathbb{R}^+$ \\ $p_z \in i\mathbb{R}^+$} & \makecell {$is$, $s \in (0, +\infty)$ \\ $\pi + is$, $s \in (0, +\infty)$ \\ when $xz \neq 0$, $xz>0$ \\ this branch is continuous at $x = 0$ \\ and `snaps' at $x = \infty$} & \makecell {$is$, $s \in (-\infty, 0)$ \\ $\pi + is$, $s \in (-\infty, 0)$ \\ when $xz \neq 0$, $xz<0$ \\ this branch is continuous at $x = 0$ \\ and `snaps' at $x = \infty$} \\
		\hline
		\makecell{$p_x \in i\mathbb{R}^+$ \\ $p_z \in \mathbb{R}^+$} & \makecell {$3\pi/2 + is$, $s \in (0, +\infty)$ \\ $\pi/2 + is$, $s \in (0, +\infty)$ \\ when $xz \neq 0$, $xz>0$ \\ this branch is continuous at $x = 0$ \\ and `snaps' at $x = \infty$} & \makecell {$\pi/2 + is$, $s \in (-\infty, 0)$ \\ $3\pi/2 + is$, $s \in (-\infty, 0)$ \\ when $xz \neq 0$, $xz<0$ \\ this branch is continuous at $x = 0$ \\ and `snaps' at $x = \infty$} \\  
		\hline      
	\end{tabular}
\end{center}
Next, for the phase shift:
\begin{center}
	\begin{tabular}{|c|c|c|}
		\hline
		&  $\theta_1$ & $\theta_2$ \\
		\hline
		$p_x \in \mathbb{R}^+, ~p_y \in \mathbb{R}^+ \Leftrightarrow \beta = \mathrm{max}$ & $\dfrac{i}{2}\ln \left|\dfrac{\sqrt{\alpha \gamma}+\sqrt{\beta \delta}}{\sqrt{\alpha \gamma}-\sqrt{\beta \delta}}\right|$ & $\dfrac{\pi}{2} + \theta_1$ \\
		\hline
		$p_y \in \mathbb{R}^+, ~p_z \in \mathbb{R}^+ \Leftrightarrow \gamma = \mathrm{max}$ & $-\dfrac{\pi}{2} + \dfrac{i}{2}\ln \left|\dfrac{\sqrt{\alpha \gamma}+\sqrt{\beta \delta}}{\sqrt{\alpha \gamma}-\sqrt{\beta \delta}}\right|$ & $\dfrac{\pi}{2} + \theta_1$\\ 
		\hline 
		$p_z \in \mathbb{R}^+, ~p_w \in \mathbb{R}^+ \Leftrightarrow \delta = \mathrm{max}$ & $ \dfrac{i}{2}\ln \left|\dfrac{\sqrt{\alpha \gamma}+\sqrt{\beta \delta}}{\sqrt{\alpha \gamma}-\sqrt{\beta \delta}}\right|$ & $-\dfrac{\pi}{2} + \theta_1$ \\
		\hline
		$p_w \in \mathbb{R}^+, ~p_x \in \mathbb{R}^+ \Leftrightarrow \alpha = \mathrm{max}$ & $ \dfrac{\pi}{2} + \dfrac{i}{2}\ln \left|\dfrac{\sqrt{\alpha \gamma}+\sqrt{\beta \delta}}{\sqrt{\alpha \gamma}-\sqrt{\beta \delta}}\right|$ & $-\dfrac{\pi}{2} + \theta_1$ \\
		\hline      
	\end{tabular}
\end{center}
\egroup

\section{Elliptic: finite solution} \label{section: planar elliptic}
This is the most general case with no extra condition on the linkage lengths. 
\begin{equation*} 
	f_{22} \neq 0, ~~ f_{20} \neq 0 , ~~ f_{02} \neq 0 ~~ \Leftrightarrow ~~ \begin{dcases}
		\alpha - \beta + \gamma - \delta \neq 0 \\
		\alpha - \beta - \gamma + \delta \neq 0 \\
		\alpha + \beta - \gamma - \delta \neq 0
	\end{dcases}
\end{equation*}
Similarly, let us start with Equation \eqref{eq: opposite rotational angle}:
\begin{equation*} 
	\begin{gathered}
		g(\alpha, ~\beta, ~\gamma, ~\delta, ~x, ~ z)=g_{22}x^2z^2+g_{20}x^2+g_{02}z^2+g_{00}=0 \\
		g_{22} = (\sigma-\alpha-\delta)(\sigma-\beta-\delta), ~~
		g_{20} = (\sigma-\alpha) (\sigma-\beta) \\
		g_{02} = -(\sigma-\gamma) (\sigma-\delta), ~~
		g_{00} = \sigma (\sigma - \alpha - \beta) \\
		\sigma=\dfrac{\alpha+\beta+\gamma+\delta}{2}
	\end{gathered}
\end{equation*}
Let
\begin{equation}
	M = \dfrac{\alpha \beta \gamma \delta}{(\sigma-\alpha)(\sigma-\beta)(\sigma-\gamma)(\sigma-\delta)} \in (0, 1) \cup (1, +\infty)
\end{equation}
If using the amplitudes $p_x$ and $p_z$, from Section \ref{section: sign convention} we could see the above equation implies:
\begin{equation} \label{eq: elliptic opposite}
	\begin{gathered}
		g(\alpha, ~\beta, ~\gamma, ~\delta, ~x, ~ z)=\left(M-1\right)\dfrac{x^2z^2}{p_x^2p_z^2}+\dfrac{x^2}{p_x^2}+\dfrac{z^2}{p_z^2} - 1 = 0\\
	\end{gathered}	
\end{equation} 
Further,
\begin{equation*}
	\begin{gathered}
		f(\alpha, ~\beta, ~\gamma, ~\delta, ~x, ~y) = f_{22} x^2y^2 + f_{20}x^2 + 2f_{11}xy + f_{02}y^2 + f_{00} = 0 \\
		f_{22} = (\sigma-\beta)(\sigma-\beta-\delta), ~~
		f_{20} = (\sigma-\alpha)(\sigma-\alpha-\delta) \\
		f_{11} = - \alpha \gamma, ~~ f_{02} =   (\sigma-\gamma)(\sigma-\gamma-\delta), ~~ f_{00}  = 
		\sigma(\sigma-\delta) \\
	\end{gathered}
\end{equation*}
Divide by $\dfrac{\sigma(\sigma-\alpha-\delta)(\sigma-\gamma-\delta)}{\sigma-\beta}$:
\begin{equation} \label{eq: elliptic adjacent}
	\begin{aligned}
		f(\alpha, ~\beta, ~\gamma, ~\delta, ~x, ~y) & = \left( \dfrac{\alpha \gamma}{(\sigma-\alpha)(\sigma-\gamma)}-1 \right)\dfrac{x^2y^2}{p_x^2p_y^2} + \dfrac{x^2}{p_x^2} + \dfrac{y^2}{p_y^2} + \dfrac{(\sigma-\beta)(\sigma-\delta)}{(\sigma-\alpha)(\sigma-\gamma)-\beta\delta} \\ & \quad - \dfrac{2\alpha \gamma}{\sqrt{(\sigma-\alpha)(\sigma-\gamma)((\sigma-\alpha)(\sigma-\gamma)-\beta\delta)}}\dfrac{xy}{p_xp_y} = 0
	\end{aligned} \\
\end{equation} 
hence 
\begin{equation} \label{eq: elliptic adjacent 2}
	\begin{aligned}
		f(\beta, ~\alpha, ~\delta, ~\gamma, ~x, ~w) & = \left( \dfrac{\beta \delta}{(\sigma-\beta)(\sigma-\delta)}-1 \right)\dfrac{x^2y^2}{p_x^2p_w^2} + \dfrac{x^2}{p_x^2} + \dfrac{w^2}{p_w^2} + \dfrac{(\sigma-\alpha)(\sigma-\gamma)}{(\sigma-\beta)(\sigma-\delta)-\alpha\beta} \\ & \quad - \dfrac{2\beta \delta}{\sqrt{(\sigma-\beta)(\sigma-\delta)((\sigma-\beta)(\sigma-\delta)-\alpha\gamma)}}\dfrac{xy}{p_xp_w} = 0
	\end{aligned} \\
\end{equation} 

It has the parametrization using the \textbf{elliptic functions} defined below. Let us start with a brief introduction. For a more complete handbook on elliptic functions we refer to this website \url{https://functions.wolfram.com/EllipticFunctions/}.

Consider a simple in-plane pendulum where a heavy object is attached to one end of a light inextensible string with length $l$. The other end of the rod is attached to a fixed point. The gravitational acceleration is $g$. The governing equation on the angular displacement $\theta$ with respect to time $t$, as we are already familiar with, has the form of 
\begin{equation*}
	\dfrac{\dif^2 \theta}{\dif t^2} + \omega^2 \sin \theta = 0, ~~\omega = \sqrt{\dfrac{g}{l}}
\end{equation*} 
If we release this pendulum at angular displacement $\alpha \in \left( 0, \dfrac{\pi}{2} \right)$, the Hamiltonian (mechanical energy) being conserved implies that:
\begin{equation*}
	\dfrac{\dif \theta}{\dif t}^2 = 2\omega^2 (\cos \theta - \cos \alpha), ~~\theta \in [-\alpha, \alpha]
\end{equation*} 
The time it takes to reach angular displacement $\phi \in [-\alpha, \alpha]$ will be
\begin{equation*}
	t\bigg|_\alpha^\phi = \dfrac{1}{\omega} \bigintsss_{0}^{\phi} \dfrac{\dif \theta }{\sqrt{2(\cos \theta - \cos \alpha)}} = \dfrac{1}{2\omega} \bigintsss_{0}^{\phi} \dfrac{\dif \theta }{\sqrt{\sin^2 \frac{\alpha}{2} - \sin^2 \frac{\theta}{2}}}
\end{equation*} 
Let 
\begin{equation*}
	\begin{gathered}
		\psi = \arcsin \left(k^{-1}\sin \frac{\theta}{2}\right), ~~\psi \in \left[-\dfrac{\pi}{2}, \dfrac{\pi}{2}\right] ,~~ k = \sin \dfrac{\alpha}{2} \in \left[ 0, \dfrac{\sqrt{2}}{2} \right]\\
	\end{gathered}
\end{equation*}
We could see that:
\begin{equation*}
	\dfrac{\dif \psi}{\dif \theta} = \dfrac{\sqrt{1-k^2 \sin^2 \psi}}{2k \cos \psi}, ~~ \dfrac{\dif \theta}{\dif \psi} = \dfrac{2k \cos \psi}{\sqrt{1-k^2 \sin^2 \psi}}
\end{equation*}
and
\begin{equation*}
	t\bigg|_\alpha^\phi = \dfrac{1}{\omega} \bigintsss_{0}^{\arcsin \left(k^{-1}\sin \frac{\phi}{2}\right)} \dfrac{\dif x}{\sqrt{1 - k^2 \sin^2 x}}
\end{equation*} 
Consider the \textit{first kind of elliptic integral} in the form below:
\begin{equation*}
	x(y) = \bigintsss_0^y \dfrac{\dif \psi}{\sqrt{1-k^2 \sin^2 \psi}}, ~~k \in (0, 1)
\end{equation*}
The quarter period $K$ is:
\begin{equation*}
	K(k) = \bigintsss_{0}^{\frac{\pi}{2}} \dfrac{\dif \psi}{\sqrt{1 - k^2 \sin^2 \psi}}
\end{equation*}  
The \textit{Jacobian elliptic functions} are:
\begin{equation*}
	\begin{gathered}
		\mathrm{sn}(x;~k) = \sin y \\
		\mathrm{cn}(x;~k) = \cos y \\
		\mathrm{dn}(x;~k) = \sqrt{1-k^2 \sin^2 y}
	\end{gathered}
\end{equation*}

\begin{figure} [!tb]
	\noindent \begin{centering}
		\includegraphics[width=1\linewidth]{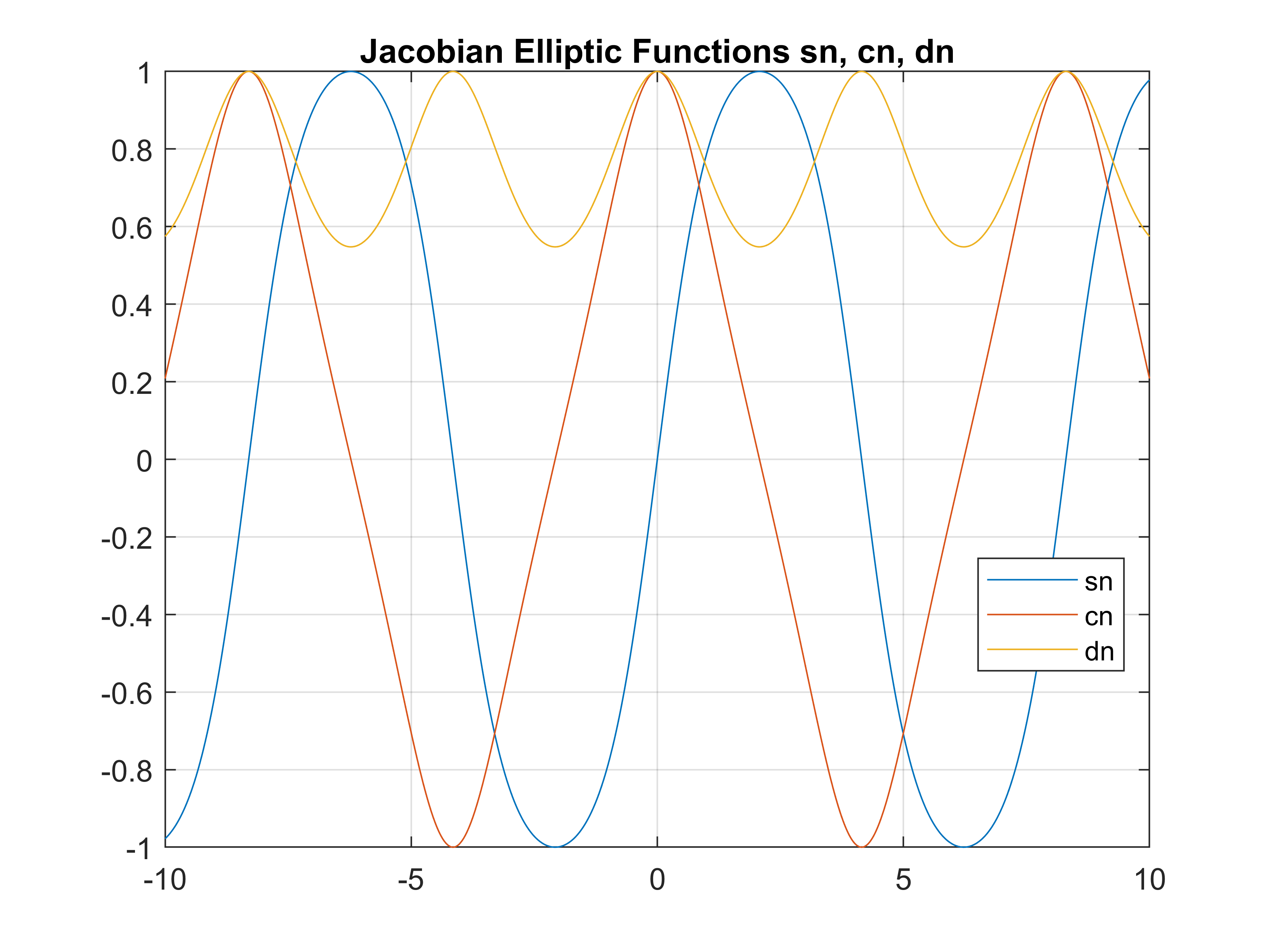}
		\par \end{centering}
	
	\caption{\label{fig: Jacobian_elliptic_function}Jacobian elliptic functions $\mathrm{sn}$, $\mathrm{cn}$, $\mathrm{dn}$ when $k^2 = 0.7$.}
\end{figure}

We list all the helpful identities here:
\begin{equation*}
	\begin{dcases}
		\mathrm{sn}(x;~0) = \sin x \\
		\mathrm{cn}(x;~0) = \cos x \\
		\mathrm{dn}(x;~0) = 1 \\
		K(0) = \dfrac{\pi}{2} 
	\end{dcases}, ~~
	\begin{dcases}
		\mathrm{sn}(x;~1) = \tanh x \\
		\mathrm{cn}(x;~1) = \dfrac{1}{\cosh x} \\
		\mathrm{dn}(x;~1) = \dfrac{1}{\cosh x} \\
		K(1) = + \infty
	\end{dcases}, ~~x \in \mathbb{R}
\end{equation*}
Let $k' = \sqrt{1-k^2}$:
\begin{equation*}
	\begin{dcases}
		\mathrm{sn}(0) = 0 \\
		\mathrm{cn}(0) = 1 \\
		\mathrm{dn}(0) = 0 \\ 
	\end{dcases}, ~~ \begin{dcases}
		\mathrm{sn}(K) = 1\\
		\mathrm{cn}(K) = 0 \\
		\mathrm{dn}(K) = k'\\ 
	\end{dcases}, ~~ 	\begin{dcases}
		\mathrm{sn}(x+2K) = - \mathrm{sn}(x)\\
		\mathrm{cn}(x+2K) = - \mathrm{cn}(x) \\
	\end{dcases}, ~~ \begin{dcases}
		\mathrm{sn}(x+4K) = \mathrm{sn}(x)\\
		\mathrm{cn}(x+4K) = \mathrm{cn}(x) \\
		\mathrm{dn}(x+2K) = \mathrm{dn}(x)\\ 
	\end{dcases}
\end{equation*}
\begin{equation*}
	\begin{dcases}
		\mathrm{sn}^2(x)+\mathrm{cn}^2(x) = 1 \\ 
		k^2\mathrm{sn}^2(x)+\mathrm{dn}^2(x) = 1
	\end{dcases}, ~~ \begin{dcases}
		\mathrm{sn}(-x) = -\mathrm{sn}(x) \\
		\mathrm{cn}(-x) = \mathrm{cn}(x) \\
		\mathrm{dn}(-x) = \mathrm{dn}(x) \\
	\end{dcases}, ~~x \in \mathbb{R}
\end{equation*}
The addition formulas:
\begin{equation*}
	\begin{dcases}
		\mathrm{sn}(x+y) = \dfrac{\mathrm{sn}(x) \mathrm{cn}(y) \mathrm{dn}(y) + \mathrm{sn}(y) \mathrm{cn}(x) \mathrm{dn}(x)}{1-k^2\mathrm{sn}^2(x) \mathrm{sn}^2(y)} \\
		\mathrm{cn}(x+y) = \dfrac{\mathrm{cn}(x) \mathrm{cn}(y) -\mathrm{sn}(x) \mathrm{sn}(y) \mathrm{dn}(x) \mathrm{dn}(y)}{1-k^2\mathrm{sn}^2(x) \mathrm{sn}^2(y)} \\
		\mathrm{dn}(x+y) = \dfrac{\mathrm{dn}(x) \mathrm{dn}(y) - k^2 \mathrm{sn}(x) \mathrm{sn}(y) \mathrm{cn}(x) \mathrm{cn}(y)}{1-k^2\mathrm{sn}^2(x) \mathrm{sn}^2(y)}
	\end{dcases}, ~~x \in \mathbb{R}
\end{equation*}
so that
\begin{equation*}
	\begin{dcases}
		\mathrm{sn}(x+K) = \dfrac{\mathrm{cn}(x)}{\mathrm{dn}(x)} \\
		\mathrm{cn}(x+K) = -\dfrac{k' \mathrm{sn}(x)}{\mathrm{dn}(x)} \\
		\mathrm{dn}(x+K) = \dfrac{k'}{\mathrm{dn}(x)}
	\end{dcases}, ~~x \in \mathbb{R}
\end{equation*}
The Jacobian elliptic functions admit an extension to the complex field, which can be computed using the addition formulas on $\mathrm{sn}(x+iy)$, $\mathrm{cn}(x+iy)$, $\mathrm{dn}(x+iy)$. 
\begin{equation*}
	\begin{dcases}
		\mathrm{sn}(ix; ~k) = i\dfrac{\mathrm{sn}(x; ~k')}{\mathrm{cn}(x; ~k')}\\
		\mathrm{cn}(ix; ~k) = \dfrac{1}{\mathrm{cn}(x; ~k')} \\
		\mathrm{dn}(ix; ~k) = \dfrac{\mathrm{dn}(x; ~k')}{\mathrm{cn}(x; ~k')}\\ 
		K' = K(k')
	\end{dcases}, ~~x \in \mathbb{R}
\end{equation*}

\begin{prop}
	Consider the parametrization of Equation \eqref{eq: elliptic opposite}, Equation \eqref{eq: elliptic adjacent} and Equation \eqref{eq: elliptic adjacent 2}.
	\begin{enumerate} [label={[\arabic*]}]
		\item When $M \in (1, +\infty)$:
		\begin{equation*}
			\begin{gathered}
				\dfrac{k^2}{1-k^2}\mathrm{cn}^2t\mathrm{cn}^2(t+K)+\mathrm{cn}^2t+\mathrm{cn}^2(t+K)-1=0 \\
				k = \sqrt{1-\dfrac{1}{M}}
			\end{gathered}
		\end{equation*}
		That is to say for Equation \eqref{eq: elliptic opposite}:
		\begin{equation*}
			x = p_x \mathrm{cn} t, ~~z = p_z \mathrm{cn} (t+K)
		\end{equation*}
		\item When $M \in (0, 1)$:
		\begin{equation*}
			\begin{gathered}
				-k^2\mathrm{sn}^2t\mathrm{sn}^2(t+K)+\mathrm{sn}^2t+\mathrm{sn}^2(t+K)-1=0 \\
				k = \sqrt{1-M}
			\end{gathered}
		\end{equation*}
		That is to say for Equation \eqref{eq: elliptic opposite}:
		\begin{equation*}
			x = p_x \mathrm{sn} t, ~~z = p_z \mathrm{sn} (t+K)
		\end{equation*}
		\item When $M \in (1, +\infty)$:
		\begin{equation*}
			\begin{gathered}
				\dfrac{k^2\mathrm{sn}^2 \theta_1}{\mathrm{dn}^2 \theta_1}\mathrm{cn}^2 t\mathrm{cn}^2 (t-\theta_1)+\mathrm{cn}^2t-\dfrac{2\mathrm{cn}\theta_1}{\mathrm{dn}^2\theta_1}\mathrm{cn} t\mathrm{cn} (t-\theta_1)+\mathrm{cn}^2(t-\theta_1)+\dfrac{(k^2-1)\mathrm{sn}^2 \theta_1}{\mathrm{dn}^2 \theta_1}=0 \\
				k = \sqrt{1-\dfrac{1}{M}}
			\end{gathered}
		\end{equation*}
		Hence for the parametrization of Equation \eqref{eq: elliptic adjacent} when $\alpha+\beta<\sigma$:
		\begin{equation*}
			\begin{dcases}
				\dfrac{k^2\mathrm{sn}^2 \theta_1}{\mathrm{dn}^2 \theta_1} = \dfrac{\alpha \gamma}{(\sigma-\alpha)(\sigma-\gamma)}-1 \\
				\dfrac{\mathrm{cn}\theta_1}{\mathrm{dn}^2\theta_1} = \dfrac{\alpha \gamma}{\sqrt{(\sigma-\alpha)(\sigma-\gamma)((\sigma-\alpha)(\sigma-\gamma)-\beta\delta)}}\\
				\dfrac{(k^2-1)\mathrm{sn}^2 \theta_1}{\mathrm{dn}^2 \theta_1} = \dfrac{(\sigma-\beta)(\sigma-\delta)}{(\sigma-\alpha)(\sigma-\gamma)-\beta\delta}
			\end{dcases}
		\end{equation*}
		The above equation determines $\theta_1$ in the form of:
		\begin{equation*}
			\mathrm{dn}\theta_1 = \sqrt{\dfrac{(\sigma-\alpha)(\sigma-\gamma)}{\alpha \gamma}}, ~~ \theta_1 \in (0, iK') \mathrm{~~when~~} \alpha + \gamma < \sigma \\
		\end{equation*}
		When $\alpha + \gamma > \sigma$, we need to do the following transformation:
		\begin{equation*}
			(\alpha, ~\beta, ~\gamma, ~\delta, ~x, ~y, ~z, ~w) \rightarrow (\delta, ~\alpha, ~\beta, ~\gamma, ~w, ~x, ~y, ~z)
		\end{equation*}
		and we could find 
		\begin{equation*}
			\mathrm{dn}\theta_1 = \sqrt{\dfrac{(\sigma-\beta)(\sigma-\delta)}{\beta \delta}}, ~~ \theta_1 \in (0, iK') \mathrm{~~when~~} \alpha + \gamma > \sigma \\
		\end{equation*}
		\item When $M \in (0, 1)$:
		\begin{equation*}
			\begin{gathered}
				- k^2 \mathrm{sn}^2 \theta_1 \mathrm{sn}^2 t\mathrm{sn}^2 (t-\theta_1)+\mathrm{sn}^2t-2\mathrm{cn}\theta_1 \mathrm{dn} \theta_1 \mathrm{sn} t\mathrm{sn} (t-\theta_1)+\mathrm{sn}^2(t-\theta_1)-\mathrm{sn}^2 \theta_1=0 \\
				k = \sqrt{1-M}
			\end{gathered}
		\end{equation*}
		Hence for the parametrization of Equation \eqref{eq: elliptic adjacent}, when $\alpha + \gamma > \sigma $:
		\begin{equation*}
			\begin{dcases}
				-k^2\mathrm{sn}^2 \theta_1 = \dfrac{\alpha \gamma}{(\sigma-\alpha)(\sigma-\gamma)}-1 \\
				\mathrm{cn}\theta_1 \mathrm{dn} \theta_1 = \dfrac{\alpha \gamma}{\sqrt{(\sigma-\alpha)(\sigma-\gamma)((\sigma-\alpha)(\sigma-\gamma)-\beta\delta)}} \\
				- \mathrm{sn}^2 \theta_1 = \dfrac{(\sigma-\beta)(\sigma-\delta)}{(\sigma-\alpha)(\sigma-\gamma)-\beta\delta} 
			\end{dcases}
		\end{equation*}
		The above equation determines $\theta_1$ in the form of:
		\begin{equation*}
			\mathrm{dn}\theta_1 = \sqrt{\dfrac{\alpha \gamma}{(\sigma-\alpha)(\sigma-\gamma)}}, ~~\theta_1 \in (0, iK') \mathrm{~~when~~} \alpha + \gamma > \sigma \\
		\end{equation*}
		When $\alpha + \gamma < \sigma$, we need to do the following transformation:
		\begin{equation*}
			(\alpha, ~\beta, ~\gamma, ~\delta, ~x, ~y, ~z, ~w) \rightarrow (\delta, ~\alpha, ~\beta, ~\gamma, ~w, ~x, ~y, ~z)
		\end{equation*}
		and we could find 
		\begin{equation*}
			\mathrm{dn}\theta_1 = \sqrt{\dfrac{(\sigma-\beta)(\sigma-\delta)}{\beta \delta}}, ~~ \theta_1 \in (0, iK') \mathrm{~~when~~} \alpha + \gamma < \sigma \\
		\end{equation*}
	\end{enumerate}
\end{prop} 
\begin{proof}
	[Proof] [1] and [2] can be verified by direct calculation. [3] and [4] could be proved by rewritten the addition formulas.
	\begin{equation*}
		\begin{gathered}
			\mathrm{cn}(t-\theta_1) = \dfrac{\mathrm{cn}(t) \mathrm{cn}(\theta_1) + \mathrm{sn}(t) \mathrm{sn}(\theta_1) \mathrm{dn}(t) \mathrm{dn}(\theta_1)}{1-k^2\mathrm{sn}^2(t) \mathrm{sn}^2(\theta_1)} \\
			\mathrm{cn}(t) = \dfrac{\mathrm{cn}(t-\theta_1) \mathrm{cn}(\theta_1) - \mathrm{sn}(t-\theta_1) \mathrm{sn}(\theta_1) \mathrm{dn}(t-\theta_1) \mathrm{dn}(\theta_1)}{1-k^2\mathrm{sn}^2(t-\theta_1) \mathrm{sn}^2(\theta_1)} \\
		\end{gathered}
	\end{equation*}
	Let $u = \mathrm{cn}t, ~v = \mathrm{cn}(t-\theta_1)$, and reform them into polynomials:
	\begin{equation*}
		\begin{gathered}
			\left([1-k^2(1-u^2)\mathrm{sn}^2 \theta_1] v - u \mathrm{cn} \theta_1 \right)^2 - (1-u^2)\mathrm{sn}^2 \theta_1(1-k^2+k^2u^2)\mathrm{dn}^2 \theta_1 = 0 \\
			\left([1-k^2(1-v^2)\mathrm{sn}^2 \theta_1] u - v \mathrm{cn} \theta_1 \right)^2 - (1-v^2)\mathrm{sn}^2 \theta_1(1-k^2+k^2v^2)\mathrm{dn}^2 \theta_1 = 0
		\end{gathered}
	\end{equation*}
	Subtract them and divide by $k^2 \mathrm{sn}^2 \theta_1 \mathrm{dn}^2 \theta_1 (u^2-v^2)$ we would obtain (3). Further,
	\begin{equation*}
		\begin{gathered}
			\mathrm{sn}(t-\theta_1) = \dfrac{\mathrm{sn} t \mathrm{cn} \theta_1 \mathrm{dn} \theta_1 - \mathrm{sn} \theta_1  \mathrm{cn} t  \mathrm{dn} t }{1-k^2\mathrm{sn}^2 t \mathrm{sn}^2 \theta_1} \\
			\mathrm{sn} t = \dfrac{\mathrm{sn} (t-\theta_1) \mathrm{cn} \theta_1 \mathrm{dn} \theta_1 + \mathrm{sn} \theta_1  \mathrm{cn} (t-\theta_1)  \mathrm{dn} (t-\theta_1) }{1-k^2\mathrm{sn}^2 (t-\theta_1) \mathrm{sn}^2 \theta_1} \\
		\end{gathered}
	\end{equation*}  
	Let $u = \mathrm{sn}t, ~v = \mathrm{sn}(t-\theta_1)$, and reform them into polynomials:
	\begin{equation*}
		\begin{gathered}
			\left([1-k^2u^2\mathrm{sn}^2 \theta_1] v - u \mathrm{cn} \theta_1 \mathrm{dn} \theta_1 \right)^2 - \mathrm{sn}^2 \theta_1(1-u^2)(1-k^2u^2) = 0 \\
			\left([1-k^2v^2\mathrm{sn}^2 \theta_1] u - v \mathrm{cn} \theta_1 \mathrm{dn} \theta_1 \right)^2 - \mathrm{sn}^2 \theta_1(1-v^2)(1-k^2v^2) = 0
		\end{gathered}
	\end{equation*}
	Subtract them and divide by $k^2 \mathrm{sn}^2 \theta_1 (u^2-v^2)$ we would obtain (4). 
\end{proof}

Similar to the derivation in Conic I, it is still necessary to clarify the magnitudes of $\alpha, ~\beta, ~\gamma, ~\delta$.
\begin{prop}
	The relation between the magnitude of $\alpha, ~ \beta, ~\gamma, ~\delta$ and the signs of amplitudes $p_x, ~p_y, ~p_z, ~p_w$. Note that $\mathrm{max} = \mathrm{max}(\alpha, ~ \beta, ~\gamma, ~\delta)$. $\mathrm{min} = \mathrm{min}(\alpha, ~ \beta, ~\gamma, ~\delta)$. 
	\begin{equation*}
		\begin{gathered}
			\mathrm{max} + \mathrm{min} > \sigma \Leftrightarrow M>1 \\
			\mathrm{max} + \mathrm{min} < \sigma \Leftrightarrow M<1
		\end{gathered}
	\end{equation*}
	\begin{enumerate} [label={[\arabic*]}]
		\item $M > 1$ and $\alpha + \gamma < \sigma$ imply $\beta = \max$ or $\delta = \max$:
		\begin{enumerate} [label={[\alph*]}]
			\item $\beta = \max \Leftrightarrow \alpha \beta > (\sigma-\alpha)(\sigma-\beta)$, $\beta \gamma > (\sigma-\beta)(\sigma-\gamma)$
			\item $\delta = \max \Leftrightarrow \alpha \beta < (\sigma-\alpha)(\sigma-\beta)$, $\beta \gamma < (\sigma-\beta)(\sigma-\gamma)$
		\end{enumerate}
		\item $M > 1$ and $\alpha + \gamma > \sigma$ imply $\alpha = \max$ or $\gamma = \max$:
		\begin{enumerate} [label={[\alph*]}]
			\item $\alpha = \max \Leftrightarrow \alpha \beta > (\sigma-\alpha)(\sigma-\beta)$, $\beta \gamma < (\sigma-\beta)(\sigma-\gamma)$
			\item $\gamma = \max \Leftrightarrow \alpha \beta < (\sigma-\alpha)(\sigma-\beta)$, $\beta \gamma > (\sigma-\beta)(\sigma-\gamma)$
		\end{enumerate}
		\item $M < 1$ and $\alpha + \gamma > \sigma$ imply $\beta = \min$ or $\delta = \min$:
		\begin{enumerate} [label={[\alph*]}]
			\item $\beta = \min \Leftrightarrow \alpha \beta < (\sigma-\alpha)(\sigma-\beta)$, $\beta \gamma < (\sigma-\beta)(\sigma-\gamma)$
			\item $\delta = \min \Leftrightarrow \alpha \beta > (\sigma-\alpha)(\sigma-\beta)$, $\beta \gamma > (\sigma-\beta)(\sigma-\gamma)$
		\end{enumerate}
		\item $M < 1$ and $\alpha + \gamma < \sigma$ imply $\alpha = \min$ or $\gamma = \min$:
		\begin{enumerate} [label={[\alph*]}]
			\item $\alpha = \min \Leftrightarrow \alpha \beta < (\sigma-\alpha)(\sigma-\beta)$, $\beta \gamma > (\sigma-\beta)(\sigma-\gamma)$
			\item $\gamma = \min \Leftrightarrow \alpha \beta > (\sigma-\alpha)(\sigma-\beta)$, $\beta \gamma < (\sigma-\beta)(\sigma-\gamma)$
		\end{enumerate}
	\end{enumerate}
\end{prop}

Now we are ready to give the solutions for both the real and complexified configuration space. 

If $M > 1$:
\begin{equation*}
	\begin{gathered}
		\begin{dcases}
			x = p_x \mathrm{cn}(t; ~k) \\
			y = p_y \mathrm{cn}(t - \theta_1; ~k) \\
			z = p_z \mathrm{cn}(t+K; ~k) \\
			w = p_w \mathrm{cn}(t - \theta_2; ~k) \\ 
		\end{dcases}, ~~ \begin{dcases}
			k = \sqrt{1-\dfrac{1}{M}}, ~~k' = \sqrt{\dfrac{1}{M}} \\
			K = \bigintsss_{0}^{\frac{\pi}{2}} \dfrac{\dif \psi}{\sqrt{1 - k^2 \sin^2 \psi}}
		\end{dcases} 
	\end{gathered}
\end{equation*}
To calculate $s$ with $x$:
\begin{equation*}
	\begin{gathered}
		\mathrm{cn}(is; ~k) = \dfrac{1}{\mathrm{cn}(s; ~k')}, ~~
		\mathrm{cn}(K + is; ~k) = -ik'\dfrac{\mathrm{sn}(s; ~k')}{\mathrm{dn}(s; ~k')} \\
	\end{gathered}
\end{equation*} 
\bgroup
\def\arraystretch{2}
\begin{center}
	\begin{tabular}{|c|c|c|}
		\hline
		\makecell{Choice of $t$ for \\ real configurations} &  Branch 1 & Branch 2 \\ 
		\hline
		\makecell{$p_x \in \mathbb{R}^+$ \\ $p_z \in i\mathbb{R}^+$} & \makecell {$is$, $s \in (0, +\infty)$ \\ $2K + is$, $s \in (0, +\infty)$ \\ when $xz \neq 0$, $xz>0$ \\ this branch `snaps' at $x = \infty$ \\ and `snaps' at $x = 0$} & \makecell {$is$, $s \in (-\infty, 0)$ \\ $2K + is$, $s \in (-\infty, 0)$ \\ when $xz \neq 0$, $xz<0$ \\ this branch `snaps' at $x = \infty$ \\ and `snaps' at $x = 0$} \\
		\hline
		\makecell{$p_x \in i\mathbb{R}^+$ \\ $p_z \in \mathbb{R}^+$} & \makecell {$3K + is$, $s \in (0, +\infty) $ \\ $K + is$, $s \in (0, +\infty) $ \\ when $xz \neq 0$, $xz>0$ \\ this branch `snaps' at $x = \infty$ \\ and `snaps' at $x = 0$} & \makecell {$K + is$, $s \in (-\infty, 0)$ \\ $3K + is$, $s \in (-\infty, 0)$ \\ when $xz \neq 0$, $xz<0$ \\ this branch `snaps' at $x = \infty$ \\ and `snaps' at $x = 0$} \\  
		\hline      
	\end{tabular}
\end{center}
for phase shift:
\begin{center}
	\begin{tabular}{|c|c|c|}
		\hline
		&  $\theta_1$ & $\theta_2$ \\
		\hline
		$p_x \in \mathbb{R}^+, ~p_y \in \mathbb{R}^+$ & $i\mathrm{dc}^{-1}\left(\sqrt{\dfrac{(\sigma - \alpha)(\sigma - \gamma)}{\alpha \gamma}};~k'\right)$ & $K + \theta_1$ \\
		\hline
		$p_y \in \mathbb{R}^+, ~p_z \in \mathbb{R}^+$ & $-K + i\mathrm{dc}^{-1}\left(\sqrt{\dfrac{(\sigma - \beta)(\sigma - \delta)}{\beta \delta}};~k'\right)$ & $K + \theta_1$\\ 
		\hline 
		$p_z \in \mathbb{R}^+, ~p_w \in \mathbb{R}^+$ & $ i\mathrm{dc}^{-1}\left(\sqrt{\dfrac{(\sigma - \alpha)(\sigma - \gamma)}{\alpha \gamma}};~k'\right)$ & $-K + \theta_1$ \\
		\hline
		$p_w \in \mathbb{R}^+, ~p_x \in \mathbb{R}^+$ & $ K + i\mathrm{dc}^{-1}\left(\sqrt{\dfrac{(\sigma - \beta)(\sigma - \delta)}{\beta \delta}};~k'\right)$ & $-K + \theta_1$ \\
		\hline      
	\end{tabular}
\end{center}
\begin{equation*}
	\begin{aligned}
		\mathrm{dn} \theta_1 = \dfrac{\mathrm{dn}(\theta_1';~k')}{\mathrm{cn}(\theta_1';~k')} = \mathrm{dc}(\theta_1';~k') \mathrm{~~when~~} \theta_1 \in (0, iK') 
	\end{aligned}
\end{equation*}
\egroup

If $M < 1$:
\begin{equation*}
	\begin{gathered}
		\begin{dcases}
			x = p_x \mathrm{sn}(t; ~k) \\
			y = p_y \mathrm{sn}(t - \theta_1; ~k) \\
			z = p_z \mathrm{sn}(t+K; ~k) \\
			w = p_w \mathrm{sn}(t - \theta_2; ~k) \\ 
		\end{dcases}, ~~ \begin{dcases}
			k = \sqrt{1-M}, ~~ k' = \sqrt{M}\\
			K = \bigintsss_{0}^{\frac{\pi}{2}} \dfrac{\dif \psi}{\sqrt{1 - k^2 \sin^2 \psi}}
		\end{dcases}
	\end{gathered}
\end{equation*}
To calculate $s$ with $x$:
\begin{equation*}
	\begin{gathered}
		\mathrm{sn}(K + is; ~k) = \dfrac{1}{\mathrm{dn}(s; ~k')}, ~~ \mathrm{sn}(is; ~k) = i \dfrac{\mathrm{sn}(s; ~k')}{\mathrm{cn}(s; ~k')}
	\end{gathered}
\end{equation*} 

\bgroup
\def\arraystretch{2}
\begin{center}
	\begin{tabular}{|c|c|c|}
		\hline
		\makecell{Choice of $t$ for \\ real configurations} &  Branch 1 & Branch 2 \\ 
		\hline
		\makecell{$p_x \in \mathbb{R}^+$ \\ $p_z \in i\mathbb{R}^+$} & \makecell {$K + is$, $s \in (0, +\infty)$ \\ $3K + is$, $s \in (0, +\infty)$ \\ when $xz \neq 0$, $xz>0$ \\ this branch `snaps' at $x = \infty$ \\ and `snaps' at $x = 0$} & \makecell {$K + is$, $s \in (-\infty, 0)$ \\ $3K + is$, $s \in (-\infty, 0)$ \\ when $xz \neq 0$, $xz<0$ \\ this branch `snaps' at $x = \infty$ \\ and `snaps' at $x = 0$} \\
		\hline
		\makecell{$p_x \in i\mathbb{R}^+$ \\ $p_z \in \mathbb{R}^+$} & \makecell {$is$, $s \in (0, +\infty)$ \\ $2K + is$, $s \in (0, +\infty)$ \\ when $xz \neq 0$, $xz>0$ \\ this branch `snaps' at $x = \infty$ \\ and `snaps' at $x = 0$} & \makecell {$2K + is$, $s \in (-\infty, 0)$ \\ $is$, $s \in (-\infty, 0)$ \\ when $xz \neq 0$, $xz<0$ \\ this branch `snaps' at $x = \infty$ \\ and `snaps' at $x = 0$} \\  
		\hline      
	\end{tabular}
\end{center}
for phase shift:
\begin{equation*}
	\begin{gathered}
		\mathrm{dn} (\theta_1; ~k) = \sqrt{\dfrac{\alpha \gamma}{(\sigma - \alpha)(\sigma - \gamma)}} ~~\mathrm{or}~~ \sqrt{\dfrac{\beta \delta}{(\sigma - \beta)(\sigma - \delta)}} \\
		\mathrm{dn} (\theta_1; ~k) = \dfrac{\mathrm{dn}(\theta_1';~k')}{\mathrm{cn}(\theta_1';~k')} = \mathrm{dc}(\theta_1';~k') \mathrm{~~when~~} \theta_1 = i \theta_1' \in (0, K') \\
	\end{gathered}
\end{equation*}
\begin{center}
	\begin{tabular}{|c|c|c|}
		\hline
		&  $\theta_1$ & $\theta_2$ \\
		\hline
		$p_x \in \mathbb{R}^+, ~p_y \in \mathbb{R}^+$ & $i\mathrm{dc}^{-1}\left(\sqrt{\dfrac{\alpha \gamma}{(\sigma - \alpha)(\sigma - \gamma)}};~k'\right)$ & $K + \theta_1$ \\
		\hline
		$p_y \in \mathbb{R}^+, ~p_z \in \mathbb{R}^+$ & $-K + i\mathrm{dc}^{-1}\left(\sqrt{\dfrac{\beta \delta}{(\sigma - \beta)(\sigma - \delta)}};~k'\right)$ & $K + \theta_1$\\ 
		\hline 
		$p_z \in \mathbb{R}^+, ~p_w \in \mathbb{R}^+$ & $ i\mathrm{dc}^{-1}\left(\sqrt{\dfrac{\alpha \gamma}{(\sigma - \alpha)(\sigma - \gamma)}};~k'\right)$ & $-K + \theta_1$ \\
		\hline
		$p_w \in \mathbb{R}^+, ~p_x \in \mathbb{R}^+$ & $ K + i\mathrm{dc}^{-1}\left(\sqrt{\dfrac{\beta \delta}{(\sigma - \beta)(\sigma - \delta)}};~k'\right)$ & $-K + \theta_1$ \\
		\hline      
	\end{tabular}
\end{center}
\egroup

The results in Elliptic I and Elliptic II are formally symmetric, and have a good correspondence with Deltoid II, Conic I, Conic II and Conic III.

\section{Orthodiagonal: finite solution}

The condition on the bar lengths is:
\begin{equation*} 
	f_{22} \neq 0, ~~ f_{20} \neq 0 , ~~ f_{02} \neq 0 ~~ \Leftrightarrow ~~ \begin{dcases}
		\alpha - \beta + \gamma - \delta \neq 0 \\
		\alpha - \beta - \gamma + \delta \neq 0 \\
		\alpha + \beta - \gamma - \delta \neq 0
	\end{dcases}
\end{equation*}
and
\begin{equation*}
	\alpha^2 + \gamma^2 = \beta^2 + \delta^2 \Leftrightarrow (\sigma - \alpha)(\sigma - \gamma) = (\sigma - \beta)(\sigma - \delta)
\end{equation*}
From the helpful identities listed in Proposition \ref{prop: identities}, we have the following simplified Equations \eqref{eq: planar 4-bar linkage} and \eqref{eq: opposite rotational angle}:
\begin{equation*}
	\begin{gathered}
		f(\alpha, ~\beta, ~\gamma, ~\delta, ~x, ~y) = f_{22} x^2y^2 + f_{20}x^2 + 2f_{11}xy + f_{02}y^2 + f_{00} = 0 \\
		f_{22} = (\sigma-\beta)(\sigma-\beta-\delta) = \dfrac{1}{2}(\beta - \alpha)(\beta - \gamma)\\
		f_{20} = (\sigma-\alpha)(\sigma-\alpha-\delta) = \dfrac{1}{2}(\beta - \alpha)(\beta + \gamma)\\
		f_{11} = - \alpha \gamma \\
		f_{02} = (\sigma-\gamma)(\sigma-\gamma-\delta) = \dfrac{1}{2}(\beta + \alpha)(\beta - \gamma) \\
		f_{00}  = 
		\sigma(\sigma-\delta) = \dfrac{1}{2}(\beta + \alpha)(\beta + \gamma) \\
		\sigma=\dfrac{\alpha+\beta+\gamma+\delta}{2}
	\end{gathered}
\end{equation*}
That is to say $x$ and $y$ are separable:
\begin{equation*}
	f(\alpha, ~\beta, ~\gamma, ~\delta, ~x, ~y) = 0 \Leftrightarrow \left( (\beta - \alpha)x + \dfrac{(\beta + \alpha)}{x}  \right) \left( (\beta - \gamma)y + \dfrac{(\beta + \gamma)}{y}  \right) =  4\alpha \gamma
\end{equation*}
Next, for opposite rotational angles:
\begin{equation*} 
	g(\alpha, ~\beta, ~\gamma, ~\delta, ~x, ~ z)=g_{22}x^2z^2+g_{20}x^2+g_{02}z^2+g_{00}=0
\end{equation*}
\begin{equation*} 
	\begin{gathered}
		g_{22}= (\sigma-\alpha-\beta)(\sigma-\beta-\gamma) = \dfrac{\gamma (\gamma-\delta)-\beta(\beta-\alpha) }{2} \\
		g_{20} =  (\sigma-\alpha) (\sigma-\beta) = \dfrac{\gamma (\gamma+\delta)-\beta(\beta-\alpha) }{2} \\
		g_{02} =  -(\sigma-\gamma) (\sigma-\delta) = \dfrac{\gamma (\gamma-\delta)-\beta(\beta+\alpha)}{2}\\
		g_{00} =  \sigma (\sigma - \alpha - \beta) = \dfrac{\gamma (\gamma+\delta)-\beta(\beta+\alpha)}{2}\\
	\end{gathered}
\end{equation*}
In the parametrized expression, we could obtain $\alpha^2 + \gamma^2 = \beta^2 + \delta^2 \Rightarrow M<1$, since
\begin{equation*}
	\begin{aligned}
		1-M = & \dfrac{-\sigma(\sigma-\alpha-\beta)(\sigma-\alpha-\gamma)(\sigma-\beta-\gamma)}{(\sigma-\alpha)(\sigma-\beta)(\sigma-\gamma)(\sigma-\delta)} \\
		& = \left(\dfrac{\alpha\gamma - \beta\delta}{\alpha\gamma + \beta\delta}\right)^2
	\end{aligned}
\end{equation*}
and
\begin{equation*}
	k = \sqrt{1-M} = \dfrac{|\alpha\gamma - \beta\delta|}{\alpha\gamma + \beta\delta} 
\end{equation*}
Further we could say that $\theta_1'$ is half of the quarter period $K'$ of elliptic functions with modulus $k'$:
\begin{equation*}
	\theta_1' = \dfrac{K'}{2}
\end{equation*}
This is because for half-quarter-periods
\begin{equation*}
	\mathrm{dc}\left(\dfrac{K'}{2}; ~k'\right) = \sqrt{1 + k}
\end{equation*}
If $\alpha\gamma > \beta\delta$, 
\begin{equation*}
	\begin{aligned}
		\mathrm{dc}\left(\dfrac{K'}{2}; ~k'\right) = \sqrt{\dfrac{2 \alpha \gamma}{\alpha\gamma + \beta\delta}} = \sqrt{\dfrac{\alpha \gamma}{(\sigma - \alpha)(\sigma - \gamma)}}
	\end{aligned}
\end{equation*}
If $\alpha\gamma < \beta\delta$, 
\begin{equation*}
	\begin{aligned}
		\mathrm{dc}\left(\dfrac{K'}{2}; ~k'\right) = \sqrt{\dfrac{2 \beta \delta}{\alpha\gamma + \beta\delta}} = \sqrt{\dfrac{\beta \delta}{(\sigma - \beta)(\sigma - \delta)}}
	\end{aligned}
\end{equation*}

\section{Rhombus: solution at infinity}

The condition on linkage lengths is 
\begin{equation*}
	f_{22}=0, ~~ f_{20}=0,~~f_{02}=0 ~~ \Leftrightarrow ~~  \alpha=\beta=\gamma=\delta
\end{equation*}
which means:
\begin{equation*}
	\begin{aligned}
		\begin{cases}
			(x_1y_1-x_2y_2)x_2y_2=0 \\
			(x_1w_1-x_2w_2)x_2w_2=0 \\
			x_1^2z_2^2-x_2^2z_1^2=0 \\
			y_1^2w_2^2-y_2^2w_1^2=0
		\end{cases} 
		~~ \Rightarrow ~~ & \begin{cases}
			x_1 \neq 0, ~~ x_2=0 \\
			z_1 \neq 0, ~~ z_2=0 \\
			w = \pm y
		\end{cases} \mathrm{or} \quad \begin{cases}
			y_1 \neq 0, ~~ y_2=0 \\
			w_1 \neq 0, ~~ w_2=0 \\
			x = \pm z
		\end{cases} \\
		~~ \Rightarrow ~~ & \begin{cases}
			x = \infty \\
			z = \infty \\
			w = \pm y
		\end{cases}  \mathrm{or} \quad \begin{cases}
			y = \infty \\
			w = \infty \\
			x =\pm z
		\end{cases}
	\end{aligned}
\end{equation*}
From post-examination, the solutions at infinity are these two branches, each of which is diffeomorphic to a circle $S^1$.
\begin{equation*}
	\begin{cases}
		x \in \mathbb{R} \cup \{\infty\} \\
		y \equiv \infty \\
		z = -x \\ 
		w \equiv \infty \\
	\end{cases} \mathrm{or} \quad \begin{cases}
		x \equiv \infty \\
		y \in \mathbb{R} \cup \{\infty\} \\
		z \equiv \infty \\ 
		w = -y \\
	\end{cases}	 
\end{equation*}

\section{Isogram: solution at infinity}

The condition on linkage lengths is 	
\begin{equation*}
	f_{22} \neq 0, ~~ f_{20}=0,~~f_{02}=0 ~~ \Leftrightarrow ~~ \gamma=\alpha, ~~ \delta=\beta, ~~\beta \neq \alpha
\end{equation*}
which means:
\begin{equation*}
	\begin{aligned}
		& \begin{cases}
			(\alpha-\beta)x_1^2y_1^2-2\alpha x_1y_1x_2y_2+(\alpha+\beta)x_2^2y_2^2 =0 \\
			x_1^2z_2^2-x_2^2z_1^2=0 \\
			(\beta-\alpha)x_1^2w_1^2-2\beta x_1w_1x_2w_2+(\beta+\alpha)x_2^2w_2^2 =0 \\
		\end{cases} \\
		~~ \Rightarrow ~~ & \begin{cases}
			x_1 \neq 0, ~~ x_2=0 \\
			y_1 = 0, ~~ y_2 \neq 0 \\
			z_1 \neq 0, ~~ z_2=0 \\
			w_1 = 0, ~~ w_2 \neq 0 \\
		\end{cases} \mathrm{or} \quad \begin{cases}
			y_1 \neq 0, ~~ y_2=0 \\
			z_1 = 0, ~~ z_2 \neq 0 \\
			w_1 \neq 0, ~~ w_2=0 \\
			x_1 = 0, ~~ x_2 \neq 0 \\
		\end{cases} \\
		~~ \Rightarrow ~~ & \begin{cases}
			x = \infty \\
			y = 0 \\
			z = \infty \\
			w = 0
		\end{cases}  \mathrm{or} \quad \begin{cases}
			x = 0 \\
			y = \infty \\
			z = 0 \\
			w = \infty \\
		\end{cases}
	\end{aligned}
\end{equation*}
These two points are exactly the limit points of the two branches in the finite solution.

\section{Deltoid I: solution at infinity}

The condition on linkage lengths is
\begin{equation*}
	f_{22}=0, ~~ f_{20} \neq 0, ~~f_{02}=0 ~~ \Leftrightarrow ~~ \delta=\alpha, ~~ \gamma=\beta, ~~\beta \neq \alpha
\end{equation*}
which means:
\begin{equation*}
	\begin{aligned}
		& \begin{cases}
			(\beta-\alpha)x_1^2y_2^2-2\alpha x_1y_1x_2y_2+(\beta+\alpha)x_2^2y_2^2 =0 \\
			x_1^2z_2^2-x_2^2z_1^2=0 \\
			(\alpha-\beta)x_1^2w_2^2-2\beta x_1w_1x_2w_2+(\alpha+\beta)x_2^2w_2^2 =0 \\
		\end{cases} \\
		~~ \Rightarrow ~~ & \begin{cases}
			y_1 \neq 0, ~~ y_2 = 0 \\
			z = \pm x	\\
			w_1 \neq 0, ~~ w_2 = 0 \\
		\end{cases}
		\Rightarrow ~~  \begin{cases}
			y = \infty \\
			z = \pm x \\
			w = \infty \\	
		\end{cases}
	\end{aligned}
\end{equation*}
From post-examination, the solutions at infinity becomes another branch diffeomorphic to a circle $S^1$:
\begin{equation*}
	\begin{cases}
		x \in \mathbb{R} \cup \{\infty\} \\
		y \equiv \infty \\
		z = -x \\ 
		w \equiv \infty \\
	\end{cases}
\end{equation*}

\section{Deltoid II: solution at infinity}

The condition on linkage lengths is 	
\begin{equation*}
	f_{22}=0, ~~ f_{20}=0 , ~~ f_{02} \neq 0 ~~ \Leftrightarrow ~~ \alpha=\beta, ~~ \delta=\gamma, ~~\gamma \neq \beta
\end{equation*}
which means:
\begin{equation*}
	\begin{aligned}
		& \begin{cases}
			(\beta-\gamma)x_2^2y_1^2-2\gamma x_1y_1x_2y_2+(\beta+\gamma)x_2^2y_2^2 =0 \\
			\beta^2 x_2^2 (z_1^2+z_2^2)=\gamma^2z_2^2(x_1^2+x_2^2) \\
			(\beta-\gamma)x_2^2w_1^2-2\gamma x_1w_1x_2w_2+(\beta+\gamma)x_2^2w_2^2 =0 \\	
			w_1^2y_2^2-w_2^2y_1^2=0 
		\end{cases} \\
		~~ \Rightarrow ~~ & \begin{cases}
			x_1 \neq 0, ~~ x_2 = 0 \\
			y = \pm w \\
			z_1 \neq 0, ~~ z_2 = 0 \\	
		\end{cases} 
		~~ \Rightarrow ~~ \begin{cases}
			x = \infty \\
			y = \pm w  \\
			z = \infty \\	
		\end{cases}
	\end{aligned}
\end{equation*}
From post-examination, the solutions at infinity becomes another branch diffeomorphic to a circle $S^1$:
\begin{equation*}
	\begin{cases}
		x \equiv \infty \\
		y \in \mathbb{R} \cup \{\infty\} \\
		z \equiv \infty \\ 
		w = -y \\
	\end{cases}	 
\end{equation*}

\section{Conic I: solution at infinity}

The condition on linkage lengths is 	
\begin{equation*} 
	f_{22}=0, ~~ f_{20} \neq 0 , ~~ f_{02} \neq 0  ~~ \Leftrightarrow ~~ \begin{dcases}
		\alpha - \beta + \gamma - \delta = 0 \\
		\alpha - \beta - \gamma + \delta \neq 0 \\
		\alpha + \beta - \gamma - \delta \neq 0
	\end{dcases}
\end{equation*} 
which implies $\sigma=\alpha+\gamma=\beta+\delta$, and:
\begin{equation*} 
	\begin{aligned}
		& ~~ \Rightarrow ~~ \begin{cases}
			\gamma(\beta-\alpha)x_1^2y_2^2-2\alpha\gamma x_1y_1x_2y_2+\alpha(\beta-\gamma)x_2^2y_1^2+\beta(\alpha+\gamma)x_2^2y_2^2 =0 \\
			\alpha \beta x_2^2 (z_1^2+z_2^2)= \gamma \delta z_2^2(x_1^2+x_2^2) \\
			\delta(\alpha-\beta)x_1^2w_2^2-2\beta\delta x_1w_1x_2w_2+\beta(\alpha-\delta)x_2^2w_1^2+\alpha(\beta+\delta)x_2^2w_2^2 =0
		\end{cases} \\
		& ~~ \Rightarrow ~~ \begin{cases}
			x_1 \neq 0, ~~ x_2=0 \\
			y_1 \neq 0, ~~ y_2 = 0 \\
			z_1 \neq 0, ~~ z_2=0 \\
			w_1 \neq 0, ~~ w_2 = 0 \\
		\end{cases} ~~ \Rightarrow ~~ \begin{cases}
			x = \infty \\
			y = \infty \\
			z = \infty \\
			w = \infty \\
		\end{cases} \\
	\end{aligned}
\end{equation*}
This is exactly the limit point of the two branches in the finite solution.

\section{Conic II: solution at infinity}

The condition on linkage lengths is 
\begin{equation*} 
	f_{22} \neq 0, ~~ f_{20} = 0 , ~~ f_{02} \neq 0 ~~ \Leftrightarrow ~~ \begin{dcases}
		\alpha - \beta + \gamma - \delta \neq 0 \\
		\alpha - \beta - \gamma + \delta = 0 \\
		\alpha + \beta - \gamma - \delta \neq 0
	\end{dcases}
\end{equation*}
\begin{equation*}
	\sigma = \dfrac{\alpha + \beta - \gamma - \delta}{2}
\end{equation*}
\begin{equation*} 
	\begin{aligned}
		& \begin{cases}
			\gamma(\alpha-\beta)x_1^2y_1^2-2\alpha\gamma x_1y_1x_2y_2+\beta(\alpha-\gamma)x_2^2y_1^2+\alpha(\beta+\gamma)x_2^2y_2^2 =0 \\
			\alpha \beta x_2^2 (z_1^2+z_2^2)=\gamma \delta z_2^2(x_1^2+x_2^2) \\
			\delta(\beta-\alpha)x_1^2w_1^2-2\beta\delta x_1w_1x_2w_2+\alpha(\beta-\delta)x_2^2w_1^2+\beta(\alpha+\delta)x_2^2w_2^2 =0 \\
		\end{cases} \\ 
		~~ \Rightarrow ~~ & \begin{cases}
			x_1 \neq 0, ~~ x_2=0 \\
			y_1 = 0, ~~ y_2 \neq 0 \\
			z_1 \neq 0, ~~ z_2=0 \\
			w_1 = 0, ~~ w_2 \neq 0 \\
		\end{cases} \mathrm{or} \quad \begin{cases}
			y_1 \neq 0, ~~ y_2 = 0 \\
			\gamma(\alpha-\beta)x_1^2+\beta(\alpha-\gamma)x_2^2=0 \\
			\alpha \beta (z^2+1)=\gamma \delta (x^2+1) \\
			\beta \gamma (w^{-2}+1) = \delta \alpha  \\
		\end{cases} \\
		& \mathrm{or} \quad \begin{cases}
			w_1 \neq 0, ~~ w_2 = 0 \\
			\delta(\beta-\alpha)x_1^2+\alpha(\beta-\delta)x_2^2=0 \\
			\delta \alpha (y^{-2} +1 ) = \beta \gamma \\
			\alpha \beta (z^2+1)=\gamma \delta (x^2+1) \\
		\end{cases} \\
		~~ \Rightarrow ~~ & \begin{cases}
			x = \infty \\
			y = 0 \\
			z = \infty \\
			w=0 
		\end{cases} \mathrm{or} \quad \begin{dcases}
			x = \pm \sqrt{ \dfrac{\alpha(\beta-\gamma)}{\gamma(\beta-\alpha)}-1} \\
			y = \infty \\
			z = \pm \sqrt{\dfrac{\delta(\gamma-\beta)}{\beta(\gamma-\delta)}-1} \\
			w^{-1} = \pm \sqrt{\dfrac{\delta \alpha}{\beta \gamma}-1}  \\
		\end{dcases} \quad \mathrm{or} \quad \begin{dcases}
			x = \pm \sqrt{ \dfrac{\beta(\alpha-\delta)}{\delta(\alpha-\beta)}-1} \\
			y^{-1} = \pm \sqrt{\dfrac{\beta \gamma}{ \delta \alpha}-1}  \\
			z = \pm \sqrt{ \dfrac{\gamma(\delta-\alpha)}{\alpha(\delta-\gamma)}-1} \\
			w = \infty \\
		\end{dcases}
	\end{aligned}
\end{equation*}

From post-examination, the isolated solutions at infinity are:
\begin{equation*}
	\begin{cases}
		x = \infty \\
		y = 0 \\
		z = \infty \\
		w=0 
	\end{cases}
\end{equation*}
and when $\delta \alpha > \beta\gamma$, 
\begin{equation*}
	\begin{dcases}
		x = \mathrm{sign} (\sigma)
		\sqrt{ \dfrac{\alpha(\beta-\gamma)}{\gamma(\beta-\alpha)}-1} \\
		y = \infty \\
		z = -\mathrm{sign} (\sigma)  \sqrt{\dfrac{\delta(\gamma-\beta)}{\beta(\gamma-\delta)}-1} \\
		w^{-1} = \sqrt{\dfrac{\delta \alpha}{\beta \gamma}-1}  \\ 
	\end{dcases} , ~ \begin{dcases}
		x = -\mathrm{sign} (\sigma) \sqrt{ \dfrac{\alpha(\beta-\gamma)}{\gamma(\beta-\alpha)}-1} \\
		y = \infty \\
		z = \mathrm{sign} (\sigma) \sqrt{\dfrac{\delta(\gamma-\beta)}{\beta(\gamma-\delta)}-1} \\
		w^{-1} = -\sqrt{\dfrac{\delta \alpha}{\beta \gamma}-1}  \\
	\end{dcases}
\end{equation*}
On the other hand, when $\delta \alpha < \beta\gamma$,
\begin{equation*}
	\begin{dcases}
		x = \mathrm{sign} (\sigma) \sqrt{ \dfrac{\beta(\alpha-\delta)}{\delta(\alpha-\beta)}-1} \\
		y^{-1} = \sqrt{\dfrac{\beta \gamma}{ \delta \alpha}-1}  \\
		z = -\mathrm{sign} (\sigma) \sqrt{ \dfrac{\gamma(\delta-\alpha)}{\alpha(\delta-\gamma)}-1} \\
		w = \infty 
	\end{dcases} , ~ \begin{dcases}
		x = -\mathrm{sign} (\sigma) \sqrt{ \dfrac{\beta(\alpha-\delta)}{\delta(\alpha-\beta)}-1} \\
		y^{-1} = - \sqrt{\dfrac{\beta \gamma}{ \delta \alpha}-1}  \\
		z = \mathrm{sign} (\sigma) \sqrt{ \dfrac{\gamma(\delta-\alpha)}{\alpha(\delta-\gamma)}-1} \\
		w = \infty \\ 
	\end{dcases}
\end{equation*}
Note that the two groups of solutions above will switch between real and complex value depending on whether $\delta \alpha > \beta\gamma$ or $\delta \alpha < \beta\gamma$. Here only real values are admissible solutions at infinity. 

\section{Conic III: solution at infinity}

The condition on linkage lengths is 
\begin{equation*} 
	f_{22} \neq 0, ~~ f_{20} \neq 0 , ~~ f_{02} = 0 ~~ \Leftrightarrow ~~ \begin{dcases}
		\alpha - \beta + \gamma - \delta \neq 0 \\
		\alpha - \beta - \gamma + \delta \neq 0 \\
		\alpha + \beta - \gamma - \delta = 0
	\end{dcases}
\end{equation*}
\begin{equation*}
	\sigma = \dfrac{-\alpha + \beta + \gamma - \delta}{2} + \pi
\end{equation*}
\begin{equation*} 
	\begin{aligned}
		& \begin{cases}
			\alpha(\gamma-\beta)x_1^2y_1^2+\beta(\gamma-\alpha)x_1^2y_2^2-2\alpha\gamma x_1y_1x_2y_2+\gamma(\alpha+\beta)x_2^2y_2^2 =0 \\
			\alpha \beta x_1^2 (z_1^2+z_2^2)=\gamma \delta z_1^2(x_1^2+x_2^2) \\
			\beta(\delta-\alpha)x_1^2w_1^2+\alpha(\delta-\beta)x_1^2w_2^2-2\beta\delta x_1w_1x_2w_2+\delta(\alpha+\beta)x_2^2w_2^2 =0 \\
		\end{cases} 
	\end{aligned}
\end{equation*}	
\begin{equation*} 
	\begin{aligned}		 
		~~ \Rightarrow ~~ & \begin{cases}
			x_1 \neq 0, ~~ x_2 = 0 \\
			\alpha(\gamma-\beta)y_1^2+\beta(\gamma-\alpha)y_2^2=0 \\
			\alpha \beta (z^{-2}+1)=\gamma \delta \\
			\beta(\delta-\alpha)w_1^2+\alpha(\delta-\beta)w_2^2=0  \\
		\end{cases} \mathrm{or} \quad  \begin{cases}
			x_1 = 0, ~~ x_2 \neq 0 \\
			y_1 \neq 0, ~~ y_2 = 0 \\
			z_1 = 0, ~~ z_2 \neq 0 \\
			w_1 \neq 0, ~~ w_2 = 0 \\
		\end{cases} \\
		& \mathrm{or} \quad \begin{cases}
			z_1 \neq 0, ~~ z_2 = 0 \\
			\alpha(\gamma-\delta)w_1^2+\delta(\gamma-\alpha)w_2^2=0  \\
			\gamma \delta (x^{-2}+1)=\alpha \beta \\
			\delta(\beta-\gamma)y_1^2+\gamma(\beta-\delta)y_2^2=0 \\
		\end{cases}
	\end{aligned}
\end{equation*}	
\begin{equation*} 
\begin{aligned}	
		~~ \Rightarrow ~~ & \begin{dcases}
			x = \infty \\
			y = \pm \sqrt{ \dfrac{\gamma(\beta-\alpha)}{\alpha(\beta-\gamma)}-1} \\
			z^{-1} = \pm \sqrt{\dfrac{\gamma \delta}{\alpha \beta}-1} \\
			w = \pm \sqrt{\dfrac{\delta(\alpha-\beta)}{\beta(\alpha-\delta)}-1} 
		\end{dcases} \quad \mathrm{or} \quad  \begin{cases}
			x = 0 \\
			y = \infty \\
			z = 0 \\
			w = \infty \\
		\end{cases} 
		\mathrm{or} \quad  \begin{dcases}
			x^{-1} = \pm \sqrt{\dfrac{\alpha \beta}{\gamma \delta}-1} \\
			y = \pm \sqrt{ \dfrac{\beta(\gamma-\delta)}{\delta(\gamma-\beta)}-1} \\
			z = \infty \\
			w = \pm \sqrt{\dfrac{\alpha(\delta-\gamma)}{\gamma(\delta-\alpha)}-1}  \\
		\end{dcases}
	\end{aligned}
\end{equation*}

From post-examination, the isolated solutions at infinity are:
\begin{equation*}
	\begin{dcases}
		x = 0 \\
		y = \infty \\
		z = 0 \\
		w = \infty \\
	\end{dcases}
\end{equation*}
and, when $\gamma \delta > \alpha \beta$:
\begin{equation*}
	\begin{dcases}
		x = \infty \\
		y = \mathrm{sign} (\sigma) \sqrt{ \dfrac{\gamma(\beta-\alpha)}{\alpha(\beta-\gamma)}-1} \\
		z^{-1} = \sqrt{\dfrac{\gamma \delta}{\alpha \beta}-1} \\
		w = -\mathrm{sign} (\sigma) \sqrt{\dfrac{\delta(\alpha-\beta)}{\beta(\alpha-\delta)}-1} 
	\end{dcases} , ~ \begin{dcases}
		x = \infty \\
		y = -\mathrm{sign} (\sigma) \sqrt{ 	\dfrac{\gamma(\beta-\alpha)}{\alpha(\beta-\gamma)}-1} \\
		z^{-1} = - \sqrt{\dfrac{\gamma \delta}{\alpha 	\beta}-1} \\
		w = \mathrm{sign} (\sigma) 	\sqrt{\dfrac{\delta(\alpha-\beta)}{\beta(\alpha-\delta)}-1} 
	\end{dcases} 
\end{equation*}
when $\gamma \delta < \alpha \beta$:
\begin{equation*}
	\begin{dcases}
		x^{-1} = \sqrt{\dfrac{\alpha \beta}{\gamma \delta}-1} \\
		y = \mathrm{sign} (\sigma) \sqrt{ \dfrac{\beta(\gamma-\delta)}{\delta(\gamma-\beta)}-1} \\
		z = \infty \\
		w = -\mathrm{sign} (\sigma) \sqrt{\dfrac{\alpha(\delta-\gamma)}{\gamma(\delta-\alpha)}-1}  \\
	\end{dcases} , ~ \begin{dcases}
		x^{-1} = - \sqrt{\dfrac{\alpha \beta}{\gamma \delta}-1} \\
		y = -\mathrm{sign} (\sigma) \sqrt{ \dfrac{\beta(\gamma-\delta)}{\delta(\gamma-\beta)}-1} \\
		z = \infty \\
		w = \mathrm{sign} (\sigma) \sqrt{\dfrac{\alpha(\delta-\gamma)}{\gamma(\delta-\alpha)}-1}  \\
	\end{dcases}
\end{equation*}
Note that the two groups of solutions above will switch between real and complex value depending on whether $\gamma \delta > \alpha \beta$ or $\gamma \delta < \alpha \beta$. Here only real values are admissible solutions at infinity. 

\section{Elliptic: solution at infinity}

The condition on linkage lengths is:
\begin{equation*} 
	f_{22} \neq 0, ~~ f_{20} \neq 0 , ~~ f_{02} \neq 0 ~~ \Leftrightarrow ~~ \begin{dcases}
		\alpha - \beta + \gamma - \delta \neq 0 \\
		\alpha - \beta - \gamma + \delta \neq 0 \\
		\alpha + \beta - \gamma - \delta \neq 0
	\end{dcases}
\end{equation*}
\begin{equation*}
	\sigma = \dfrac{\alpha + \beta +\gamma + \delta}{2}
\end{equation*}
\begin{equation*}
	\begin{aligned}
		& \begin{cases}
			x_1 \neq 0, ~~ x_2 = 0 \\
			(\sigma-\beta)(\sigma-\delta-\beta)y_1^2+(\sigma-\alpha)(\sigma-\delta-\alpha)y_2^2=0 \\
			(\sigma-\alpha)(\sigma-\gamma-\alpha)w_1^2+(\sigma-\beta)(\sigma-\gamma-\beta)w_2^2=0 \\
			(\sigma-\beta-\delta)(\sigma-\alpha-\delta)z_1^2+(\sigma-\beta)(\sigma-\alpha)z_2^2=0
		\end{cases} \\
		\mathrm{or} \quad & \begin{cases}
			y_1 \neq 0, ~~ y_2 = 0 \\
			(\sigma-\gamma)(\sigma-\alpha-\gamma)z_1^2+(\sigma-\beta)(\sigma-\alpha-\beta)z_2^2=0 \\
			(\sigma-\beta)(\sigma-\delta-\beta)x_1^2+(\sigma-\gamma)(\sigma-\delta-\gamma)x_2^2=0 \\
			(\sigma-\gamma-\alpha)(\sigma-\beta-\alpha)w_1^2+(\sigma-\gamma)(\sigma-\beta)w_2^2=0
		\end{cases} \\
		\mathrm{or} \quad & \begin{cases}
			z_1 \neq 0, ~~ z_2 = 0 \\
			(\sigma-\delta)(\sigma-\beta-\delta)w_1^2+(\sigma-\gamma)(\sigma-\beta-\gamma)w_2^2=0 \\
			(\sigma-\gamma)(\sigma-\alpha-\gamma)y_1^2+(\sigma-\delta)(\sigma-\alpha-\delta)y_2^2=0 \\
			(\sigma-\delta-\beta)(\sigma-\gamma-\beta)x_1^2+(\sigma-\delta)(\sigma-\gamma)x_2^2=0
		\end{cases} \\
		\mathrm{or} \quad & \begin{cases}
			w_1 \neq 0, ~~ w_2 = 0 \\
			(\sigma-\alpha)(\sigma-\gamma-\alpha)x_1^2+(\sigma-\delta)(\sigma-\gamma-\delta)x_2^2=0 \\
			(\sigma-\delta)(\sigma-\beta-\delta)z_1^2+(\sigma-\alpha)(\sigma-\beta-\alpha)z_2^2=0 \\
			(\sigma-\alpha-\gamma)(\sigma-\delta-\gamma)y_1^2+(\sigma-\alpha)(\sigma-\delta)y_2^2=0
		\end{cases} \\
	\end{aligned}
\end{equation*}

\begin{equation*}
	\begin{aligned}
		\Rightarrow \quad & \begin{dcases}
			x =  \infty \\
			y =  \pm \sqrt{\dfrac{\gamma(\beta-\alpha)}{(\sigma-\beta)(\sigma-\alpha-\gamma)}-1} \\
			z^{-1} = \pm \sqrt{\dfrac{\gamma \delta}{(\sigma-\alpha)(\sigma-\beta)}-1} \\
			w = \pm \sqrt{\dfrac{\delta(\alpha-\beta)}{(\sigma-\alpha)(\sigma-\beta-\delta)}-1} 
		\end{dcases}  ~\mathrm{or}~ \begin{dcases}
			y =  \infty \\
			z =  \pm \sqrt{\dfrac{\delta(\gamma-\beta)}{(\sigma-\gamma)(\sigma-\beta-\delta)}-1} \\
			w^{-1} = \pm \sqrt{\dfrac{\delta \alpha}{(\sigma-\beta)(\sigma-\gamma)}-1} \\
			x = \pm \sqrt{\dfrac{\alpha(\beta-\gamma)}{(\sigma-\beta)(\sigma-\alpha-\gamma)}-1}
		\end{dcases} \\
		\mathrm{or} ~ & \begin{dcases}
			z = \infty \\
			w = \pm \sqrt{\dfrac{\alpha(\delta-\gamma)}{(\sigma-\delta)(\sigma-\alpha-\gamma)}-1} \\
			x^{-1} = \pm \sqrt{\dfrac{\alpha \beta}{(\sigma-\gamma)(\sigma-\delta)}-1} \\
			y = \pm \sqrt{\dfrac{\beta(\gamma-\delta)}{(\sigma-\gamma)(\sigma-\beta-\delta)}-1} \\
		\end{dcases}
		~\mathrm{or}~ \begin{dcases}
			w =  \infty \\
			x =  \pm \sqrt{\dfrac{\beta(\alpha-\delta)}{(\sigma-\alpha)(\sigma-\beta-\delta)}-1} \\
			y^{-1} = \pm \sqrt{\dfrac{\beta \gamma}{(\sigma-\alpha)(\sigma-\delta)}-1} \\
			z = \pm \sqrt{\dfrac{\gamma(\delta-\alpha)}{(\sigma-\delta)(\sigma-\alpha-\gamma)}-1} \\
		\end{dcases} \\
	\end{aligned} 
\end{equation*} 

Similarly, our next step is to determine how many of the above solutions are real. Take $ x = \infty$ as an example, the condition for $y, ~z, ~w$ to be real numbers is (referring to the helpful identities in Subsection \ref{section: sign convention}): 
\begin{equation*}
	\begin{aligned}
		(\sigma-\alpha)(\sigma-\beta) - \gamma \delta & = (\sigma-\alpha-\gamma)(\sigma-\beta-\delta) \\
		& = \dfrac{\alpha \beta \gamma \delta - (\sigma-\alpha)(\sigma-\beta)(\sigma-\gamma)(\sigma-\delta)}{\sigma(\sigma-\alpha-\beta)} < 0
	\end{aligned}
\end{equation*} 
Following the notation introduced for the elliptic type used in Section \ref{section: planar elliptic}, 
\begin{equation}
	M = \dfrac{\alpha \beta \gamma \delta}{(\sigma-\alpha)(\sigma-\beta)(\sigma-\gamma)(\sigma-\delta)} \in (0, 1) \cup (1, +\infty)
\end{equation}
the condition for $x=\infty$ to be a real configuration is
\begin{equation*}
	(M-1)(\sigma-\alpha-\beta) < 0
\end{equation*} 
Let us write all the cases together:
\begin{equation*}
	\begin{dcases}
		(M-1)(\sigma-\alpha-\beta)<0 ~~ \Leftrightarrow ~~ x=\infty \mathrm{~can~be~reached} \\
		(M-1)(\sigma-\beta-\gamma)<0 ~~ \Leftrightarrow ~~ y=\infty \mathrm{~can~be~reached} \\
		(M-1)(\sigma-\gamma-\delta)<0 ~~ \Leftrightarrow ~~ z=\infty \mathrm{~can~be~reached} \\
		(M-1)(\sigma-\delta-\alpha)<0 ~~ \Leftrightarrow ~~ w=\infty \mathrm{~can~be~reached} \\
	\end{dcases}
\end{equation*}

From post-examination, there are 8 solutions at infinity, when $(M-1)(\sigma-\alpha-\beta)<0$, $x$ reaches the infinity:
\begin{equation*}
	\begin{dcases}
		x =  \infty \\
		y =  \mathrm{sign} (\sigma) \sqrt{\dfrac{\gamma(\beta-\alpha)}{(\sigma-\beta)(\sigma-\alpha-\gamma)}-1} \\
		z^{-1} = \sqrt{\dfrac{\gamma \delta}{(\sigma-\alpha)(\sigma - \beta)}-1} \\
		w = -\mathrm{sign} (\sigma) \sqrt{\dfrac{\delta(\alpha-\beta)}{(\sigma-\alpha)(\sigma-\beta-\delta)}-1} 
	\end{dcases} , ~ 
	\begin{dcases}
		x =  \infty \\
		y =  -\mathrm{sign} (\sigma) \sqrt{\dfrac{\gamma(\beta-\alpha)}{(\sigma-\beta)(\sigma-\alpha-\gamma)}-1} \\
		z^{-1} = - \sqrt{\dfrac{\gamma \delta}{(\sigma-\alpha)(\sigma-\beta)}-1} \\
		w = \mathrm{sign} (\sigma) \sqrt{\dfrac{\delta(\alpha-\beta)}{(\sigma-\alpha)(\sigma-\beta-\delta)}-1} 
	\end{dcases} \\
\end{equation*}
When $(M-1)(\sigma-\beta-\gamma)<0$, $y$ reaches the infinity:
\begin{equation*}
	\begin{dcases}
		y =  \infty \\
		z =  \mathrm{sign} (\sigma) \sqrt{\dfrac{\delta(\gamma-\beta)}{(\sigma-\gamma)(\sigma-\beta-\delta)}-1} \\
		w^{-1} = \sqrt{\dfrac{\delta \alpha}{(\sigma-\beta)(\sigma-\gamma)}-1} \\
		x = -\mathrm{sign} (\sigma) \sqrt{\dfrac{\alpha(\beta-\gamma)}{(\sigma-\beta)(\sigma-\alpha-\gamma)}-1} \\
	\end{dcases} , ~
	\begin{dcases}
		y =  \infty \\
		z =  -\mathrm{sign} (\sigma) \sqrt{\dfrac{\delta(\gamma-\beta)}{(\sigma-\gamma)(\sigma-\beta-\delta)}-1} \\
		w^{-1} = - \sqrt{\dfrac{\delta \alpha}{(\sigma-\beta)(\sigma-\gamma)}-1} \\
		x = \mathrm{sign} (\sigma) \sqrt{\dfrac{\alpha(\beta-\gamma)}{(\sigma-\beta)(\sigma-\alpha-\gamma)}-1} \\
	\end{dcases}
\end{equation*}
When $(M-1)(\sigma-\gamma-\delta)<0$, $z$ reaches the infinity:
\begin{equation*}
	\begin{dcases}
		z =  \infty \\
		w = \mathrm{sign} (\sigma) \sqrt{\dfrac{\alpha(\delta-\gamma)}{(\sigma-\delta)(\sigma-\alpha-\gamma)}-1} \\
		x^{-1} = \sqrt{\dfrac{\alpha \beta}{(\sigma-\gamma)(\sigma-\delta)}-1} \\
		y = -\mathrm{sign} (\sigma) \sqrt{\dfrac{\beta(\gamma-\delta)}{(\sigma-\gamma)(\sigma-\beta-\delta)}-1} \\
	\end{dcases} , ~ 
	\begin{dcases}
		z =  \infty \\
		w = -\mathrm{sign} (\sigma) \sqrt{\dfrac{\alpha(\delta-\gamma)}{(\sigma-\delta)(\sigma-\alpha-\gamma)}-1} \\
		x^{-1} = - \sqrt{\dfrac{\alpha \beta}{(\sigma-\gamma)(\sigma-\delta)}-1} \\
		y = \mathrm{sign} (\sigma) \sqrt{\dfrac{\beta(\gamma-\delta)}{(\sigma-\gamma)(\sigma-\beta-\delta)}-1} \\
	\end{dcases}
\end{equation*}
When $(M-1)(\sigma-\delta-\alpha)<0$, $w$ reaches the infinity:
\begin{equation*}
	\begin{dcases}
		w =  \infty \\
		x =  \mathrm{sign} (\sigma) \sqrt{\dfrac{\beta(\alpha-\delta)}{(\sigma-\alpha)(\sigma-\beta-\delta)}-1} \\
		y^{-1} = \sqrt{\dfrac{\beta \gamma}{(\sigma-\alpha)(\sigma-\delta)}-1} \\
		z = -\mathrm{sign} (\sigma) \sqrt{\dfrac{\gamma(\delta-\alpha)}{(\sigma-\delta)(\sigma-\alpha-\gamma)}-1} \\
	\end{dcases} ,~ 
	\begin{dcases}
		w =  \infty \\
		x =  -\mathrm{sign} (\sigma) \sqrt{\dfrac{\beta(\alpha-\delta)}{(\sigma-\alpha)(\sigma-\beta-\delta)}-1} \\
		y^{-1} = -\sqrt{\dfrac{\beta \gamma}{(\sigma-\alpha)(\sigma-\delta)}-1} \\
		z = \mathrm{sign} (\sigma) \sqrt{\dfrac{\gamma(\delta-\alpha)}{(\sigma-\delta)(\sigma-\alpha-\gamma)}-1} \\
	\end{dcases} 
\end{equation*}

\section{The Grashof condition and self-intersection}

\begin{prop}
	(The Grashof condition) If the sum of the shortest and longest linkage of a planar quadrilateral linkage is less than or equal to the sum of the remaining two linkages, then the shortest linkage can rotate fully with respect to a neighbouring linkage.
\end{prop}

We could naturally examine this condition from our previous discussions on the solutions at infinity. It holds obviously except for the elliptic type. Without loss of generality, suppose $\beta$ is the shortest linkage, then the condition for $x= \infty$ or $y= \infty$ is
\begin{equation*}
	(M-1)(\sigma-\alpha-\beta)<0 ~~ \mathrm{or} ~~ (M-1)(\sigma-\beta-\gamma)<0 
\end{equation*}
Since $\beta$ is the shortest, the statement below holds
\begin{equation*}
	\sigma-\alpha-\beta>0 ~~\mathrm{or}~~ \sigma-\beta-\gamma>0
\end{equation*}
hence in order for $x= \infty$ or $y= \infty$ can be reached, we need to require:
\begin{equation*}
	M < 1 \Leftrightarrow \max + \min < \sigma
\end{equation*}
The right hand side is exactly the Grashof condition.

\begin{prop}
	When no rotational angle reaches the infinity, a planar 4-bar linkage is self-intersected if and only if
	\begin{equation*}
		(\mathrm{sign}(x), \mathrm{sign}(y), \mathrm{sign}(z), \mathrm{sign}(w)) = 
		\begin{dcases}
			(1, ~1, ~-1, ~-1) \\
			(-1, ~1, ~1, ~-1) \\
			(-1, ~-1, ~1, ~1) \\
			(1, ~-1, ~-1, ~1)
		\end{dcases}
	\end{equation*}  
\end{prop}

\begin{proof}
	Consider two line segments, say $AB$ and $CD$, These segments intersect if and only if the points $A$ and $B$ lie on opposite sides of the line $CD$ and the points $C$ and $D$ lie on the opposite sides of the line $AB$. The points $A$ and $B$ lie on different sides of $CD$ if and only if the oriented areas of the triangles $ACD$ and $BCD$ have different signs. So for a planar 4-bar linkage $ABCD$ it suffices to let the oriented areas of all four triangles be the pattern list on the right hand side of the above proposition. Further, the signs of these oriented areas equal to the signs of the rotational angles $x, ~y, ~z, ~w$.
\end{proof}

We have introduced this self-intersection check in the accompanying MATLAB app \citep{he_elliptic-fun-based_2023}.

\section*{Acknowledgement}

We thank Prof. Ivan Izmestiev for personal communication on checking the self-intersection from the oriented area. We thank James Drummond for his assistance in the MATLAB app design. This work is supported by the MathWorks-CUED grant.

\bibliographystyle{plainnat}

\end{document}